\documentclass[12pt,leqno,a4paper]{amsart}
\usepackage{fullpage}
\usepackage{color,amsmath,amsthm,amscd,amssymb,colonequals,enumerate}
\usepackage[english]{babel}
\usepackage{tikz}
\usetikzlibrary{cd}
\usetikzlibrary{decorations.markings}
\usepackage{blkarray} 
\usepackage{verbatim}

\theoremstyle{plain}
\newtheorem{thm}{Theorem}[section]
\newtheorem{prop}[thm]{Proposition}
\newtheorem{lem}[thm]{Lemma}
\newtheorem{cor}[thm]{Corollary}

\theoremstyle{remark}
\newtheorem*{rem}{Remark} 
\theoremstyle{definition}
\newtheorem{exam}[thm]{Example}

\newtheorem{hyps}[thm]{Hypotheses}

\numberwithin{equation}{subsection}

\let\nc\newcommand
\let\dmo\DeclareMathOperator
\dmo{\Spec}{Spec}
\dmo{\Hom}{Hom}
\dmo{\Pic}{Pic}
\dmo{\coker}{coker}
\dmo{\im}{im}
\dmo{\Div}{Div}
\dmo{\ord}{ord}
\dmo{\Gal}{Gal}
\dmo{\Aut}{Aut}
\dmo{\Frac}{Frac}

\nc\cG{\mathcal G}
\nc\cT{\mathcal T}
\nc\cX{\mathcal X}
\nc\cJ{\mathcal J}
\nc\cL{\mathcal L}
\nc\cO{\mathcal O}
\nc\cI{\mathcal I}
\nc\norm{\mathcal N}
\nc\R{\mathcal R}
\nc\cF{\mathcal{F}}
\nc\fm{\mathfrak m}
\nc\ZZ{\mathbb{Z}}
\nc\QQ{\mathbb{Q}}
\nc\QQbar{\overline{\QQ}}
\nc\Gm{\mathbb{G}_{\mathrm m}}
\nc\FFp{\mathbb{F}_p}
\nc\FFpbar{\overline{\mathbb{F}}_p}
\nc\XX{\mathbb X}
\nc\et{\mathrm{\acute et}}
\nc\etale{\'etale}
\nc\supp{\mathrm{supp}}
\nc\sing{\mathrm{sing}}
\nc\lin{\mathrm{lin}}
\nc\unip{\mathrm{unip}}
\nc\tor{\mathrm{tor}}
\nc\red{\mathrm{red}}
\nc\id{\mathrm{id}}
\nc\pr{\mathrm{pr}}
\nc\diag{\mathrm{diag}}
\nc\sm{\mathrm{sm}}
\nc\sn{\mathrm{sn}}
\nc\reg{\mathrm{reg}}
\nc\ks{k^\mathrm{sep}}
\nc\Fsep{\F^{\mathrm{sep}}}
\nc\cGm{\cG_\mathrm m}
\nc\Fsh{\F^{\mathrm{sh}}}
\nc\Fhatsh{\widehat\F^{\mathrm{sh}}}
\nc\Ssh{S^{\mathrm{sh}}}
\renewcommand\Rsh{R^{\mathrm{sh}}}

\nc\Ytilde{\widetilde Y}
\nc\Gtilde{\widetilde\Gamma}
\nc\Itilde{\widetilde I}
\nc\Ctilde{\widetilde C}
\nc\itilde{\tilde i}
\nc\kbar{\bar k}     
\nc\isom{\xrightarrow\sim}
\nc\longisom{\xrightarrow{\ \sim\ }}
\nc\mmu{\mbox{$\raisebox{-0.59ex}
  {$l$}\hspace{-0.18em}\mu\hspace{-0.88em}\raisebox{-0.98ex}{\scalebox{2}
  {$\color{white}.$}}\hspace{-0.416em}\raisebox{+0.88ex}
  {$\color{white}.$}\hspace{0.46em}$}{}}
\nc\defeq{\colonequals} 
\nc\mul{\times} 
\nc\vzero{v_0} 
\nc\divides{\mid}
\nc\notdivides{\nmid}
\nc\laplace{\square}
\nc\F{F} 
\nc\ch[1]{\XX(#1)} 
\nc\comp[1]{\Phi(#1)} 
\nc\kModulus{\Sigma} 
\nc\ka{\kappa}

\nc\xp{x'} 
\nc\Fp{\F'} 
\nc\Ner{\mbox{\upshape N\'er}} 
\nc\e[1]{e_{#1}}
\let\setminus\smallsetminus
\nc\zero{_0} 
\nc\pspl{{p\textup{-spl}}}
\nc\M{L} 
\nc\dual[1]{#1^\vee} 
\nc\SPS{\mathsf{SS}}
\nc\SPSM{\SPS_{\hspace*{-0.01in}M}}
\nc\SPSMp{\SPS_{\hspace*{-0.01in}M'}}
\nc\C[2][d]{C_{#1,#2}}
\DeclareMathOperator{\SL}{SL}
\nc\noop[1]{} 

\nc\new[1]{#1}
\nc\mpa{}
\dedicatory{To the memory of Bas Edixhoven}

\usepackage{microtype}

\usepackage[lite]{amsrefs} 

\title{Modular curves and N\'{e}ron models of generalized Jacobians}
\usepackage[
	pdfauthor={Bruce W. Jordan, Kenneth A. Ribet, and Anthony J. Scholl},
        pdftitle={Modular curves and N\'{e}ron models of generalized
          Jacobians},
        hypertexnames=false
        ]{hyperref}

\begin{document}

\subjclass[2020]{Primary 14K30; Secondary 11G18, 14G35}
\keywords{N\'{e}ron models, generalized Jacobians, modular curves,
  cuspidal modulus}

\author{Bruce~W.~Jordan}
\address{Department of Mathematics, Baruch College, The City University
of New York, One Bernard Baruch Way, New York, NY 10010-5526, USA}
\email{bruce.jordan@baruch.cuny.edu}

\author{Kenneth~A.~Ribet}
\address{Mathematics Department, University of California, Berkeley, CA
94720, USA}
\email{ribet@math.berkeley.edu}

\author{Anthony~J.~Scholl}
\address{Department of Pure Mathematics and Mathematical Statistics, 
Centre for Mathematical Sciences, Wilberforce Road, Cambridge CB3 0WB, 
England}
\email{a.j.scholl@dpmms.cam.ac.uk}
\begin{abstract}
  Let $X$ be a smooth geometrically connected projective curve over
  the field of fractions of a discrete valuation ring $R$, and $\fm$ a
  modulus on $X$, given by a closed subscheme of $X$ which is
  geometrically reduced. The generalized Jacobian $J_\fm$ of $X$ with
  respect to $\fm$ is then an extension of the Jacobian of $X$ by a
  torus.  We describe its N\'eron model, together with the character
  and component groups of the special fibre, in terms of a regular model of
  $X$ over $R$. This generalizes Raynaud's well-known description for
  the usual Jacobian. We also give some computations for generalized
  Jacobians of modular curves $X_0(N)$ with moduli supported on the
  cusps.
\end{abstract}
\maketitle

\subsection*{Introduction}

Let $R$ be a discrete valuation ring, with field of fractions $\F$ and
residue field $k$. Let $\cX$ be a regular scheme, proper and flat over
$S=\Spec R$, whose generic fibre $X=\cX_\F$ is a smooth curve. In
\cite{Ray70} Raynaud describes the relationship between the N\'eron
model of the Jacobian $J=\Pic^0_{X/\F}$ of $X$ and the relative Picard
functor $P=\Pic_{\cX/S}$. The aim of this paper is twofold: first, to
extend Raynaud's results to the generalized Jacobian $J_\fm$ of $X$
with respect to a reduced modulus $\fm$. Secondly, to apply these
results to compute the component and character groups of the N\'eron
models of generalized Jacobians attached to modular curves and moduli
supported on cusps.

Our motivation for this work arises from
applications to the arithmetic of modular forms --- the point being
that just as the arithmetic of cusp forms of weight $2$ on a
congruence subgroup of $\SL(2,\ZZ)$ is controlled by the Jacobian of
the associated complete modular curve, so the arithmetic of the space
of holomorphic modular forms on the same group is controlled by a
suitable generalized Jacobian. Raynaud's results have been used
extensively to study the arithmetic of cusp forms of weight $2$ and
their associated Galois representations --- for example, in
\cites{Maz77,MW84,Rib88,Rib90}. In future work we plan to give
arithmetic applications of the results obtained here. We note that
generalized modular Jacobians with cuspidal modulus are
considered in Gross \cite{Gros12},  Yamazaki and Yang
\cite{YY16}, Bruinier and Li \cite{BL16}, Wei and Yamazaki
\cite{WY19}, and Iranzo \cite{Ira21}.
Another point of view, using $1$-motives rather than generalized
Jacobians (see also Section \ref{sec:suzuki} below), has been
investigated by Lecouturier \cite{Lec21}.

Before describing our main results, we briefly recall from \cite{Ray70}
the results of Raynaud on Jacobians. To simplify the discussion, we
assume for the rest of this introduction that $R$ is Henselian, $k$ is
algebraically closed, and that the greatest common divisor of the
multiplicities of the irreducible components of the fibre $\cX_s$
at the closed point $s=\Spec k$ is $1$. (We review Raynaud's theory in \S2.2--3
below in greater detail and under less restrictive hypotheses.)  Under
these hypotheses, \cite[(8.2.1)]{Ray70} shows that $P$ is represented by
a smooth group scheme over $S$, and there is a canonical morphism of
group schemes $\deg\colon P \to \ZZ$, which maps a line bundle to its
total degree along the fibres of $\cX/S$. The open and closed subgroup
scheme $P'=\ker(\deg)$ then has $J$ as its generic fibre.

Let $r$ be the number of irreducible components of $\cX_s$. If $r>1$
then $P$ is not separated over $S$. Indeed, if $Y\subset \cX_s^\red$
is an irreducible component, viewed as a reduced divisor on $\cX$,
then the line bundle $\cO_\cX(Y)$ represents an element of $P'(S)$,
nonzero if $r>1$, whose image in $P'(\F)$ vanishes. The closure
$E\subset P'$ of the zero section is then an \etale\ (but not separated)
$S$-group scheme, whose generic fibre is trivial, and whose special
fibre is isomorphic to $\ZZ^{r-1}$, generated by the classes of the
bundles $\cO_\cX(Y)$ restricted to $\cX_s$. Raynaud shows:
\begin{enumerate}[i)]
\item The maximal separated quotient $P'/E$ is the N\'eron model
  $\cJ$ of $J$.
\item The identity component $\cJ_s^0$ of the special fibre of
  $\cJ$ is canonically isomorphic to the Picard scheme
  $\Pic_{\cX_s/k}^0$.
\item Let $\cJ_s^{0,\lin}$ be the maximal connected affine subgroup
  scheme of $\cJ_s^0$.  Its character group
  $\ch{J_s} \defeq \Hom_k(\cJ_s^{0,\lin},\Gm)$ is canonically isomorphic to
  $H_1(\Gtilde_{\cX_s},\ZZ)$, the integral homology of the extended
  dual graph $\Gtilde_{\cX_s}$ of the singular curve $\cX_s$ (we
  recall the definition in \S\ref{sec:singular} below).
\item The component group $\comp{J}\defeq \cJ_s/\cJ_s^0$ is
  canonically isomorphic to the homology of the complex
  \[
    \ZZ[C] \to \ZZ^C \to \ZZ
  \]
  where $C$ is the set of irreducible components $Y\subset
  \cX_s^\red$, the first map is given by the intersection pairing
  $C\times C \to \ZZ$ on $\cX$, and the second by $(m_Y)_Y\mapsto
  \sum_Y \delta_Ym_Y$, where $\delta_Y$ is the multiplicity of $Y$ in
  the fibre.
\end{enumerate}
In the special case where $\cX_s$ is a reduced
divisor on $\cX$ with normal crossings, iii) and iv) become:
\begin{itemize}
\item[iii${}'$)] $\Hom_k(\cJ_s^{0,\lin},\Gm)\simeq H_1(\Gamma_{\cX_s},\ZZ)$, where
   $\Gamma_{\cX_s}$ is the reduced dual graph of $\cX_s$, whose
   vertex set is $C$ and edge set is $\cX_s^\sing$.
 \item[iv${}'$)]
   $\comp{J}\simeq \coker\left(\laplace\colon \ZZ[C] \to
     \ZZ[C]\zero \right)$, where $\laplace=\laplace_0$ is the
   $0$-Laplacian (as in \cite{JRS22})  of the graph
   $\Gamma_{\cX_s}$, which is the endomorphism of $\ZZ[C]$ taking
   a vertex $v\in C$ to $\sum (v)-(v')$,  
where the sum is taken over all edges joining $v$ to an adjacent vertex $v'$.
\end{itemize}

Now let $\fm$ be a modulus (effective divisor) on $X$. Then one has
\cites{Ros54,Ser84}
the generalized Jacobian $J_\fm$ of $X$ relative to $\fm$, which is an
extension of $J$ by a commutative connected linear group $H$. Assume
that $\fm=\sum_{i\in I}(x_i)$ is a sum of distinct points, whose
residue fields $\F_i$ are all separable over $\F$. This is equivalent
to assuming that $H$ is a torus. Then by the results of Raynaud
\cite[Chapter 10]{BLR90}, $J_\fm$ has a
N\'eron model $\cJ_\fm$, which is a smooth separated group scheme over
$S$, not necessarily of finite type, with generic fibre $J_\fm$ and
satisfying the N\'eron universal property. (In the terminology of
\cite{BLR90}, $\cJ_\fm$ is an lft-N\'eron model.) We obtain results
analogous to (i)--(iv${}'$) for $\cJ_\fm$. Specifically, let $R_i$ be the
integral closure of $R$ in $\F_i$, and $\Sigma_s$ be the disjoint
union of the $\Spec( R_i\otimes_Rk)$, $i\in I$. The inclusion of the set of
points $x_i$ in $X$ gives a morphism $\Sigma_s\to\cX_s$. We show:
\begin{enumerate}[i)]
\item There exists a smooth $S$-group scheme $P_\fm$, parametrizing
  equivalence classes of line bundles on $\cX$ with a trivialisation
  at each $x_i$, and $\cJ_\fm$ is the maximal separated quotient of
  $P_\fm'=\ker\left(\deg\colon P_\fm \to \ZZ\right)$ (Theorems \ref{thm:Pm_eq_Fstar}
  and \ref{thm:main}).
\item The identity component $\cJ_{\fm,s}^0$ of the special fibre of
  $\cJ_\fm$ is canonically isomorphic to $\Pic_{(\cX_s,\Sigma_s)/k}^0$,
  the generalized Picard scheme classifying line bundles on $\cX_s$
  of degree zero on each irreducible component, together with a
  trivialisation of the pullback to $\Sigma_s$
  (Corollary \ref{cor:charGenJac}(\ref{cor:charGenJac1})).
\item The character group $\Hom_k(\cJ_{\fm,s}^{0,\lin},\Gm)$ is the integral
  homology of an extended graph $\Gtilde_{\cX_s,\Sigma}$, depending only
  on the combinatorics of the components of $\cX_s$ and the reductions
  $\bar x_i\in\cX_s$ of the points $x_i$
  (Corollary \ref{cor:charGenJac}(\ref{cor:charGenJac1})).
\item The component group $\comp{J_\fm}=\cJ_{\fm,s}/\cJ_{\fm.s}^0$,
  which is an abelian group of finite type (not necessarily finite), is
  isomorphic to the homology of the complex \eqref{eq:raynaud2}
  \[
    \ZZ[C] \oplus \ZZ \to \ZZ^C \oplus \ZZ^I \to \ZZ 
  \]
  (Theorem \ref{thm:phi_m}).
\end{enumerate}
If $\cX_s$ is a reduced divisor with normal crossings, and the points
$x_i$ are $\F$-rational, then the character and component groups have
simple descriptions in terms of the homology and Laplacian of a
generalized reduced dual graph (Corollary \ref{cor:semistable}).

We then apply these results to a modular curve $X_0(N)$ and a modulus $\fm$
supported on the cusps. If $p>3$ is a prime exactly dividing $N$,
we compute the character and component groups, together with the
action of the Hecke operators on them. In particular, if $N=p$ and
$\fm=(\infty)+(0)$ is the sum of the two cusps of $X_0(p)$, then the
component group is infinite cyclic, with $T_\ell$ acting by $\ell+1$
for $\ell\ne p$, and the representation of the full Hecke algebra on the
character group is given by the classical Brandt matrices. We also compute
the component group for $N=p^2$, which for the full cuspidal modulus
is free of rank $2$.

There has been considerable interest in ``Jacobians of graphs''
--- for example, Lorenzini \cites{Lor89,Lor91}, Bacher--de la
Harpe--Nagnibeda \cite{BdlHN97} and Baker--Norine \cite{BN07}. Our
results here on $\comp{J_\fm}$ suggest that there is also a theory of
``generalized Jacobians of graphs''. We investigate this in the paper
\cite{JRS22}.

Let us briefly describe the contents of the rest of the paper. In
Section 1, we prove our results on N\'eron models of generalized
Jacobians. Although not needed for the applications we have in mind,
we decided to work in a very general setting (in particular, there are
no conditions imposed on the base discrete valuation ring). Sections
\ref{sec:stuff}, \ref{sec:singular} and \ref{sec:raynaud} review
well-known facts about N\'eron models, Weil restriction, and Picard
schemes of singular curves, as well as some of Raynaud's results from
\cite{Ray70}.

In \S\ref{sec:singular2} and \ref{sec:functoriality1} we describe the
structure of the generalized Picard scheme of a singular curve with
respect to a modulus, and discuss its functoriality. The main
results on the N\'eron models of generalized Jacobian are contained in
\S\ref{sec:hard-bit}. In the following two sections we explain the
relation with $1$-motives, and describe some of the behaviour of the
N\'eron model of $J_\fm$ under correspondences.

In Section 2 we apply our results to the modular curves $X_0(N)$
and cuspidal moduli, computing in several cases the component and
characters groups of the reduction of the N\'eron model modulo a prime
$p>3$. 

We describe some prior work on these topics. If the points $(x_i)$ are
$F$-rational and their closures in $\cX$ are disjoint, then by
identifying them, one obtains a singular relative curve $\cX/\fm$
which is semifactorial. Some of our results in this case are then
subsumed by the works \cites{Ore17a,Ore17b,Pep13} on Picard schemes of
semifactorial curves. In \cite{Ove21}, Overkamp proves general results
on the existence of N\'eron models of Picard schemes of singular
curves. Finally,
Suzuki \cite{Suz19} has defined N\'eron models of $1$-motives and studied
their duality properties and component groups. We discuss its relation
with the present work in Section \ref{sec:suzuki}.

\subsection*{Notation}

Throughout the paper, unless otherwise stated, $R$ will denote a
discrete valuation ring with field of fractions $\F$, uniformiser
$\varpi$, and residue field $k$. Except where stated otherwise, we
make no further hypotheses on $R$ or $k$. We write
$p = \max(1, \mathrm{char}(k))$ for the characteristic exponent of
$k$. We put $S = \Spec R$ and denote by $s$ its closed point. Let
$\Rsh$ be a strict henselisation of $R$, and $\Fsh$ its field of
fractions. Write $\ks$ for the residue field of $\Rsh$ (a separable
closure of $k$), and $\bar s$ for its spectrum. We write $(Sm/S)$ for
the category of essentially smooth $S$-schemes, and $(Sm/S)_\et$ for
its \etale\ site. For a scheme $X$, we write $\kappa(x)$ for the
residue field at a point $x\in X$, and if $X$ is irreducible,
$\kappa(X)$ for the residue field of the generic point of $X$.  All
group schemes considered in this paper will be commutative. We
frequently identify \etale\ group schemes over a field with their
associated Galois modules.

If $S$ is a finite set we write $\ZZ[S]$ for the free abelian group on
$S$ and $\ZZ[S]\zero$ for the kernel of the degree map $\ZZ[S] \to \ZZ$,
$s \mapsto 1$ for $s\in S$.

\section{N\'eron models of generalized Jacobians}

\subsection{Preliminaries}
\label{sec:stuff}

In this section we collect together properties of N\'eron models and
Weil restriction of scalars. Most of these may be found in \cite{BLR90},
especially Chapter 10.

Recall that if $G/F$ is a smooth group scheme of finite type, then a
N\'eron model for $G$ is a smooth separated group scheme $\cG/S$ with
generic fibre $G$, such that for every smooth $S$-scheme $S'$, the canonical map
$\cG(S') \to \cG(S'_F)=G(S'_F)$ is bijective. If $\cG$ exists, it is unique
up to unique isomorphism. (In \cite{BLR90} these are called N\'eron
lft-models.) 
The identity component $\cG^0$ of $\cG$ is a smooth group
scheme of finite type. The formation of N\'eron models commutes with
strict henselisation and completion of the base ring $R$. If
$G \otimes_\F \Fhatsh$ does not contain a copy of
$\mathbb{G}_{\mathrm a}$ then $G$ has a N\'eron model
\cite[10.2 Thm.2]{BLR90}. (More generally, this holds if $S$ is merely a
semilocal Dedekind scheme.)  We write $\comp{G}$ for the component group
$(\cG_s/\cG_s^0)(\ks)$. If $k$ is perfect, then by Chevalley's Theorem
\cite{Con} $\cG_s$ has a unique
maximal connected affine smooth subgroup scheme $\cG_s^{0,\lin}$, and we then
write $\ch{G}$ for the character group
$\Hom(\cG_s^{0,\lin}\otimes_k\kbar,\Gm)$, a finite free $\ZZ$-module with
a continuous action of $\Gal(\kbar/k)$.

Let $0 \to G_1 \to G_2 \to G_3 \to 0$ be an exact sequence of smooth
connected $F$-groups which have N\'eron models $\cG_i$. Consider the
complexes
\begin{gather}
  0 \to \cG_1 \to \cG_2 \to \cG_3 \to 0 \label{eq:neron-exact}\\
  0 \to \cG_1^0 \to \cG_2^0 \to \cG_3^0 \to 0\label{eq:neron0-exact}\\
  0 \to \comp{G_1} \to \comp{G_2} \to \comp{G_3} \to 0\label{eq:phi-exact}
\end{gather}
The following two exactness results are a restatement of
\cite[Remark (4.8)(a)]{Cha00}, with the same proof, which we give for the reader's
convenience.

\begin{lem}
\label{lem:neron-exact}
  Suppose that the induced map $\cG_2 \to \cG_3$ is a surjection of
  sheaves for the smooth topology. Then:
\begin{enumerate}[\upshape (a)]
\item
\label{lem:neron-exact1}
The sequence \eqref{eq:neron-exact} is exact.
\item 
\label{lem:neron-exact2}
If $\comp{G_1}$ is torsion-free, then the sequences
  \eqref{eq:neron0-exact} and \eqref{eq:phi-exact} are exact.
\end{enumerate}
\end{lem}

\begin{proof} 
  (\ref{lem:neron-exact1}) Since locally for the smooth topology the morphism
  $\cG_2 \to \cG_3$ of group schemes has a section, it is evidently
  surjective. Let $\cG'$ denote its kernel. By \cite[Lemma 4.3(b)]{LL01},
  $\cG'$ is smooth. The canonical morphism
  $\cG_1 \to \cG_2$ factors though a morphism
  $\gamma \colon \cG_1 \to \cG'$ which is the identity on generic
  fibres, and since $\cG_1$ is a N\'eron model, there is a morphism
  $\delta\colon \cG' \to \cG_1$ which is the identity on generic
  fibres. As $\cG_1$ and $\cG'$ are separated over $S$, $\gamma$ and
  $\delta$ are mutually inverse isomorphisms.

  (\ref{lem:neron-exact2}) The map $\cG_2^0 \to \cG_3^0$ is surjective, so we have an
  exact sequence
  \[
    0 \to \cG_1 \cap \cG_2^0 \to \cG_2^0 \to \cG_3^0 \to 0
  \]
  in which each term is of finite type over $S$. Hence $\cG_{1,s}^0$
  has finite index in $\cG_{1,s} \cap \cG_{2,s}^0$, and since
  $\comp{G_1}$ is torsion-free we have $\cG_1 \cap \cG_2^0 =
  \cG_1^0$. So \eqref{eq:neron0-exact} and therefore also
  \eqref{eq:phi-exact} are exact.
\end{proof}

\begin{cor}
  \label{cor:tori-exact}
  Suppose that $G_1$ is a product of tori of the form
  $\R_{\F'/F} T$, where $\F'/F$ is finite separable, $T$ is an
  $\F'$-torus which splits over an unramified extension, and
  $\R_{\F'/\F}$ is Weil restriction of scalars. Then
  \eqref{eq:neron-exact}, \eqref{eq:neron0-exact} and
  \eqref{eq:phi-exact} are exact.
\end{cor}

\begin{proof}
  Replacing $R$ by $\Rsh$, we may assume that each $T/\F'$ is
  split. According to \cite[4.2]{BX96}, \cite[(4.5)]{Cha00}, one then has
  $R^1j_{\sm\,*} G_1 = 0$, where
  $j_\sm \colon (\Spec \F )_\sm \to S_\sm$ is the inclusion of small
  smooth sites. Therefore $\cG_2 \to \cG_3$ is surjective as a map of
  sheaves on $S_\sm$ . By Proposition \ref{prop:tori-neron}(a) below,
  $\comp{G_1}$ is torsion-free, so everything follows from the lemma.
\end{proof}

We will need the following minor generalization of a result from
\cite{BLR90}.

\begin{prop}
  \label{prop:converse-to-neron-exact}
  Let
  \[
    0 \to \cG_1 \to \cG_2 \to \cG_3 \to 0
  \]
  be an exact sequence of smooth $S$-group schemes. If $\cG_1$ and
  $\cG_3$ are the N\'eron models of their generic fibres, the same is
  true for $\cG_2$.
\end{prop}
\begin{proof}
  This follows by the same argument as in the proof of \S7.5,
  Proposition 1(b) in \cite{BLR90} (middle of p.185), using the
  criterion of \S10.1, Proposition 2.
\end{proof}

From \cite[§7.6]{BLR90} we recall basic properties of Weil
restriction. Let $Z'/Z$ be a finite flat morphism of finite
presentation. If $Y$ is a quasiprojective $Z'$-scheme, then the Weil
restriction $\R_{Z'/Z} Y$ exists, and is characterised by its functor
of points $\R_{Z'/Z}Y(-) = Y(- \times_Z Z')$. If $Y$ is smooth over
$Z'$ then $\R_{Z'/Z} Y$ is smooth over $Z$. If $Y \to X$ is a closed
immersion of quasiprojective $Z'$-schemes, then
$\R_{Z'/Z} Y \to \R_{Z'/Z} X$ is a closed immersion. If $Z' \to Z$ is
surjective and $Y$ is a quasiprojective $Z$-scheme, then the canonical
map $Y \to \R_{Z'/Z} (Y \times_Z Z')$ is a closed immersion.

Now let $k$ be a field, $k'$ a finite $k$-algebra, and $k''$ a finite flat $k'$-algebra. Let $Y$
be a quasiprojective $k$-scheme. There is then a canonical map
\[
g \colon \R_{k'/k} (Y \otimes_k k' ) \to \R_{k''/k} (Y \otimes_k k'').
\]
We may write $k' = k_1' \times k_2'$ where
$\Spec k_1'\subset \Spec k'$ is the image of $\Spec k''$ (and $k_2'$
is possibly zero). The morphism $g$ then factors\mpa
\begin{multline*}
\R_{k'/k} (Y \otimes_k k' ) = \R_{k_1'/k} (Y \otimes_k k_1' )
\times_{\Spec k} \R_{k_2'/k} (Y \otimes_k k_2' ) \xrightarrow{\pr_1} \\
\R_{k_1'/k} (Y \otimes_k k_1' ) \to \R_{k_1'/k} \R_{k''/k_1'} (Y
\otimes_k k'' ) = \R_{k''/k} (Y \otimes_k k'' )
\end{multline*}
and the second arrow is a closed immersion. In particular, if $Y$ is a smooth
$k$-group, then $g$ is a surjection onto a closed subgroup scheme, and its cokernel
is smooth.

Let $k$ be a field and $k'$ a finite $k$-algebra.  Then $\R_{k'/k}\Gm$ is
a connected smooth $k$-group scheme of finite type. It is a torus if
and only if
$k'/k$ is \etale.

We return to N\'eron models. Recall that the multiplicative group
$\Gm/\F$ has a N\'eron model $\cGm /S$, whose special fibre is
$\Gm \times \ZZ$. It fits into an exact sequence of group schemes
\[
0 \to \Gm \to \cGm \xrightarrow{v_\F} s_*\ZZ \to 0
\]
where on $R$-points $v_\F$ is the normalised valuation
$v_\F \colon \cGm (R) = F^\mul \to\mkern-15mu\to\ZZ$.

Let $\F'$ be a finite \etale\ $\F$-algebra, $R' \subset \F'$ the
normalisation of $R$ in $\F'$, $S' = \Spec R'$. Let
$\F' \otimes_\F \Fsh = \prod_{i\in I}
\F_i$, where the fields $\F_i$ are totally ramified extensions of
$\Fsh$, of degrees $e_i p^{s_i}$, where $p^{s_i}$ is the degree of the
(purely inseparable) residue class extension.
\begin{prop}
  \label{prop:tori-neron}\mbox{}
  \begin{enumerate}[\upshape (a)] 
  \item The N\'eron model of $\R_{\F'/\F}\Gm$ is $\R_{S'/S}\cGm$, and the product of the valuations\mpa
    \begin{equation} \label{eq:prod-of-vals}
      \R_{S'/S}\cGm(\Fsh)=\prod_{i\in I} \F_i^\mul
      \xrightarrow{(v_{\F_i})} \ZZ^I
    \end{equation}
    induces an isomorphism
    \begin{equation} \label{eq:pi0-of-tori}
      \comp{\R_{\F'/\F}\Gm}=\pi_0((\R_{S'/S}\cGm)_s) \longisom \ZZ^I.
    \end{equation}
  \item The adjunction map $\cGm \to \R_{S'/S}\cGm$ is a closed
    immersion, and its cokernel is the N\'eron model of
    $(\R_{\F'/\F}\Gm)/\Gm$, inducing a isomorphism
    \[
      \comp{(\R_{\F'/\F}\Gm)/\Gm} \longisom \coker(e=(e_i)\colon \ZZ\to\ZZ^I).
    \]
  \end{enumerate}
\end{prop}
Here we use $\cGm$ to denote also the N\'eron model of $\Gm$ over the semilocal base
$S'$.

\begin{proof}
  (a) The first statement follows from \cite{BLR90}, Propositions 10.1/4
  and 6.  For the second, replacing $\F$ by $\Fsh$ we are reduced to
  the case of a totally ramified field extension $\F'/\F$. Then as
  $\cG_{{\mathrm m},s}\simeq \mathbb{G}_{\mathrm m,s}\times \ZZ$, we
  have
  $(\R_{S'/S}\cGm)_s = \R_{R'\otimes k/k}\Gm \times \R_{R'\otimes
    k/k}\ZZ$. As the first factor is connected, and the second is
  $\ZZ$ (since $R'\otimes k/k$ is radicial) we get
  $\comp{\R_{\F'/\F}\Gm}\simeq\ZZ$, and the fact that this isomorphism
  is given by the valuation follows from \cite[1.1/Proposition 7]{BLR90}.

  (b) By Corollary \ref{cor:tori-exact}, the exact sequence
  $0\to\Gm\to\R_{\F'/\F}\Gm \to (\R_{\F'/\F}\Gm)/\Gm\to 0$ gives rise
  to exact sequences of N\'eron models and component groups. So it is
  enough to show that the map \mpa
  $\comp{\Gm}=\ZZ \to \comp{\R_{\F'/\F}\Gm}=\ZZ^I$ is equal to
  $e$. Replacing $\F$ by $\Fsh$ again, we are reduced to the case when
  $\F'/\F$ is a totally ramified field extension of degree $ep^s$ with
  residue degree $p^s$.  Then by (a) we have a commutative square
\[
\begin{tikzcd}
  \cGm \arrow[r, hook] \arrow[d, "v_{\F}"] &
  \R_{R'/R}\cGm \arrow[d, "v_{\F'}"] \\
  \ZZ \arrow[r, "e"] & \ZZ
\end{tikzcd}
\]
proving the result.
\end{proof}

\subsection{Graphs and Picard schemes of singular curves}
\label{sec:singular}

In this section we work over an arbitrary field $k$. By a curve
over $k$ we shall mean a $k$-scheme $X$ of finite type which is
equidimensional of dimension 1 and Cohen-Macaulay (i.e.,~has no
embedded points). Let $\{X_j\}$ be the irreducible components of $X$,
and $\eta_j$ the generic point of $X_j$. The local ring
$\cO_{X,\eta_j}$ is Artinian, and following Raynaud
\cite[(6.1.1) and (8.1.1)]{Ray70} we write $d_j$ for its length, and $\delta_j$ for the
total multiplicity of $X_j$ in $X$. If $k'/k$ is a radicial closure of
$k$, and $\eta'_j\in X\otimes k'$ is the point lying over $\eta_j$,
then $\delta_j$ equals the length of the local ring of
$\eta'_j$. Moreover
$\delta_j = d_j[\ka(\eta_j)\cap k' : k]=d_jp^{n_j}$ for some
$n_j\ge 0$.

Until the end of this section, $k$ denotes an algebraically closed
field. We review the well-known description of the toric part of the
Picard scheme of a singular curve over $k$.

Let $Y/k$ be a reduced proper curve, and $Y^\sing\subset Y(k)$ its set
of singular points. Write $\phi \colon \Ytilde \to Y$ for its
normalisation. Define sets
\[
A=Y^\sing\subset Y(k),\quad B=\phi^{-1}(Y^\sing)\subset \tilde
Y(k),\quad C=\pi_0(\Ytilde).
\]
We have maps
\[
\phi \colon B \to A,\quad \psi\colon B \to C
\]
where $\psi$ maps $x\in B$ to the connected component of $\Ytilde$
containing it.

The \emph{extended graph} $\Gtilde_Y = (\widetilde V,\widetilde E)$
of $Y$ is the graph with vertices $\widetilde V$ and edges $\widetilde
E$, where
\begin{itemize}
\item $\widetilde V= A\sqcup C$, $\widetilde E = B$.
\item The endpoints of an edge $b\in B$ are $\phi(b)\in A$ and
  $\psi(b)\in C$. 
\end{itemize}
The graph $\Gtilde_Y$ is bipartite, and therefore has a canonical
structure of directed graph, by directing the edge $b$ so that its
source is $\phi(b)$.

Suppose $Y$ only has double points (meaning that if $y\in Y^\sing$
then $\phi^{-1}(y)$ has exactly 2 elements). The \emph{reduced graph}
$\Gamma_Y=(V,E)$ is the undirected graph (possibly with multiple edges
and loops) whose vertex set is $V=\pi_0(\Ytilde)$ and edge set is
$E=Y^\sing$. It is obtained from $\tilde \Gamma_Y$ by, for each vertex
$v\in A$, deleting $v$ and replacing the two edges incident to $v$
with a single edge. There is a canonical homeomorphism between the
geometric realisations of $\Gtilde_Y$ and $\Gamma_Y$, under which
$v\in A$ is mapped to the midpoint of the replacing edge. If $Y/k$ is
a proper curve, not necessarily reduced, we define
$\Gtilde_Y=\Gtilde_{Y^\red}$, $\Gamma_Y=\Gamma_{Y^\red}$, where
$Y^\red\subset Y$ is the reduced subscheme.

Let $G=\Pic^0_Y$ be the identity component of the Picard scheme of
$Y$. It is a smooth group scheme of finite type over $k$, classifying
line bundles on $Y$ whose restriction to each irreducible component
has degree zero. The filtration of $G$ by its linear and unipotent
subgroups is described as follows.

Let $Y'\to Y$ be the ``seminormalisation'' of $Y$, which is obtained
from $Y$ by replacing its singularities with singularities which are
\etale\ locally isomorphic to the union of coordinate axes in
$\mathbb{A}^N_k$. The normalisation map factors into a pair of finite
morphisms $\Ytilde \overset{\phi'}\to Y' \to Y$. These give rise to a
commutative diagram, whose rows are exact:
\[
\begin{tikzcd}
  0 \arrow[r] & G^\unip \arrow[r] \arrow[d] & G=\Pic^0_Y \arrow[r] \arrow[d, equal] &
  \Pic^0_{Y'}  \arrow[r] \arrow[d, "\phi^{\prime*}"] & 0
\\
0 \arrow[r] & G^\lin \arrow[r] &  G \arrow[r] & \Pic^0_{\Ytilde} \arrow[r] &  0
\end{tikzcd}
\]
(where $G^\unip$ is the maximal connected unipotent subgroup of $G$)
giving an isomorphism $\ker\phi^{\prime*}\simeq G^\tor=G^\lin/G^\unip$
by the snake lemma.

To give a line bundle on $Y'$ is equivalent to giving a line bundle on
$\Ytilde$ together with descent data for $\phi'\colon \Ytilde\to Y'$,
so the toric part $G^\tor $ classifies trivial line bundles on
$\Ytilde$ equipped with descent data to $Y'$. For the trivial bundle
$\cO_{\Ytilde}$, to give such descent data is equivalent to giving,
for each singular point $y\in Y$, an element of
$(k^\mul)^{\phi^{-1}(y)}/k^\mul$. The automorphism group of
$\cO_{\Ytilde}$ is $(k^\mul)^{\pi_0(\Ytilde)}$. Hence $G^\tor $ is
canonically
\[
\Gm^{\pi_0(\Ytilde)} \backslash \prod_{y\in Y^{\mathrm{sing}}} 
 \Bigl(\Gm^{\phi^{-1}(y)}/\mathrm{diag}(\Gm)\Bigl).
\]
Here $\Gm^{\pi_0(\Ytilde)}$ acts on $\Gm^{\phi^{-1}(y)}$ by the dual
of the map $\phi^{-1}(y)\subset \Ytilde(k)\overset\psi\to \pi_0(Y)$
associating to $x\in \Ytilde$ the connected component of $\Ytilde$
containing it. The character group of $G^\tor$ is therefore the kernel
of the map
\[
\ZZ[B] \xrightarrow{(\psi,\phi)} \ZZ[C]\oplus\ZZ[A]
\]
which (after replacing $\phi$ with $-\phi$) is the chain complex of
$\Gtilde_Y$.  This gives the formula \cite[I.3]{DeRa72}
\begin{equation}
\label{eq:Pictor}
  \Hom(G^\lin ,\Gm) = H_1(\Gtilde_Y,\ZZ).
\end{equation}
Suppose now that $Y$ is a proper curve over $k$, not necessarily
reduced. The map $\Pic^0_Y \to \Pic^0_{Y^\red}$ is an epimorphism, and
its kernel is a connected unipotent group scheme, so \eqref{eq:Pictor}
remains valid.

If $k$ is merely assumed to be perfect, \eqref{eq:Pictor} holds as an
isomorphism of $\Gal(\kbar/k)$-modules.

If $Y$ only has double points with distinct branches, then by the
homeomorphism $\Gtilde_Y\to \Gamma_Y$ we obtain the formula
\begin{equation}
\label{eq:Pictorred}
  \Hom(G^\lin ,\Gm) = H_1(\Gamma_Y,\ZZ).
\end{equation}

\subsection{The N\'eron model of $J$}
\label{sec:raynaud}

In preparation for \S\ref{sec:hard-bit}, we review in more detail the
results of Raynaud. We will follow mainly the notation of \cite{Ray70}
(see also \cite{BLR90}, where the notations are slightly different).

We consider a proper flat morphism $\cX \to S=\Spec R$, satisfying the
hypotheses (H1--3) below.
\begin{itemize}
\item[(H1)]
  The generic fibre $X\defeq \cX_\F$ is a smooth geometrically connected curve
  over $\F$ (in particular, $\Gamma(\cX,\cO_\cX)=R$).
\item[(H2)]
 $\cX$ is regular.
\end{itemize}
Let the irreducible components of $\cX_s$ be indexed by the set $C$,
and for $j\in C$, let $\cX_j\subset \cX_s$ be the scheme-theoretic
closure of the corresponding maximal point of $\cX_s$,
$\delta_j=p^{n_j}d_j$ its total multiplicity (\S\ref{sec:singular}), and
$Y_j= \cX_j^\red$. Define $\delta=\gcd\{\delta_j\}$,
$d=\gcd\{d_j\}$.
\begin{itemize}
\item[(H3)] $(\delta,p)=1$.
\end{itemize}
\noindent Hypotheses (H1) and
(H2) imply that Raynaud's condition $(\mathrm N)^*$ is satisfied
\cite[(6.1.4)]{Ray70}. Hypothesis (H1) is not particularly restrictive, since one may always reduce to this
case using Stein factorization. In the presence of (H1--2), hypothesis
(H3) implies that $\cX/S$ is cohomologically flat (equivalently, that
$\Gamma(\cX_s,\cO_{\cX_s})=k$), by \cite[(7.2.1)]{Ray70}.

Let $J=\Pic^0_{X/\F}$ be the Jacobian
variety of $X$, and let $\cJ$ be the N\'eron model of $J$.

The relative Picard functor $P=\Pic_{\cX/S}$ is the sheafification
(for the fppf topology) of the functor on the category of $S$-schemes
\[
S' \mapsto \Pic(\cX\times_S S').
\]
There is a morphism of abelian sheaves $\deg\colon P \to \ZZ$ which
takes a line bundle to its total degree along the fibres, and
$P'\subset P$ denotes its kernel. By \cite[(5.2) and (2.3.2)]{Ray70}, $P$
and $P'$ are formally smooth algebraic spaces over $S$, and the
closure $E\subset P$ of the zero section is an \etale\ algebraic space
over $S$, contained in $P'$.  The maximal separated quotient $Q=P/E$
is a smooth separated $S$-group scheme, and the subgroup $Q'=P'/E$ is
the closure in $Q$ of the identity component $Q^0$ (proof of
\cite[(8.1.2)(iii)]{Ray70}). One also has the subgroup $Q^\tau\subset Q$,
which is the inverse image of the torsion subgroup of $Q/Q^0$. As
$\cX$ is regular, condition d) of \cite[(8.1.2)]{Ray70} holds, and so
$Q^\tau$ is closed in $Q$. By definition $\deg(Q^\tau)=0$, and
therefore $Q'=Q^\tau$. So \cite[(8.1.2) and (8.1.4)(b)]{Ray70} imply that
$\cJ=Q'=P'/E$.

(If (H3) is not satisfied, then $P$ is in general not representable,
but it still has a maximal separated quotient $Q$ which is a smooth
separated $S$-group scheme \cite[(4.1.1)]{Ray70}. If moreover $k$ is
perfect, then $Q'$ again equals $\cJ$ \cite[(8.1.4)(a)]{Ray70}.)

Let $P_s^0$ be the identity component of $P_s$. We have
$P_s^0=\Pic^0_{\cX_s/k}$, the identity component of the Picard scheme
of $\cX_s$. By \cite[(6.4.1)(3)]{Ray70}, the intersection $P^0_s\cap E_s$
is a constant group scheme over $k$, cyclic of order $d$, generated by
the class of the line bundle $\cL'=\cO(\sum_j(d_j/d)Y_j)$. (Because
$\cX$ is regular, the integers $d$ and $d'$ \cite[(6.1.11)(3)]{Ray70} are
equal.)  Therefore $\cJ_s^0$ is canonically isomorphic to
$\Pic^0_{\cX_s/k}/\langle\cL'\rangle$, and in particular, if $d=1$ then
$\cJ_s^0=\Pic^0_{\cX_s/k}$.

Suppose that $k$ is perfect and $d=1$. Combining the above with the discussion in
\S2.2, we then have an isomorphism of $\Gal(\bar k/k)$-modules
\[
\ch{J}\defeq \Hom(\cJ_s^{0,\lin}\otimes_k\bar k, \Gm) = H_1(\tilde\Gamma_{\cX_s\otimes\bar k},\ZZ).
\]
Finally, we recall the description of the component group. First
suppose that $R$ is strictly Henselian ($k$ not necessarily
perfect). Then \cite[(8.1.2)]{Ray70} shows that the component group
$\comp{J}=\cJ_s/\cJ_s^0$ is computed as follows: by the above,
$\comp{J}= Q'_s/Q^0_s$ is the cokernel of the map
\[
  E_s \to  P'_s/P^0_s =\ker(\deg\colon P_s/P^0_s \to \ZZ).
\]
One has an isomorphism
\[
  P_s/P_s^0 \simeq \ZZ^C,\qquad (\cL\in P_s)
  \mapsto (\deg \cL|_{Y_j})_j.
\]
Let $D\subset \Div\cX$ be the group of Cartier divisors
supported in the special fibre, and $D_0\subset D$ the subgroup of
principal divisors. By \cite[(6.1.3)]{Ray70} one has $E_s=D/D_0$.  As
$\cX$ is regular and $R=\Gamma(\cX,\cO_\cX)$,  $D$ is
freely generated by the set of reduced components $\{Y_j\}$, and $D_0$
is the subgroup generated by the divisor $(\varpi)$ of the special
fibre. The complex of \cite[(8.1.2)(i)]{Ray70} then becomes
\begin{equation}
  \label{eq:raynaud}
0 \to \ZZ \xrightarrow{i}  \ZZ[C] \xrightarrow{a} \ZZ^C \xrightarrow{b} \ZZ \to 0
\end{equation}
where the maps are:
\begin{align}
  i(1) &= \sum_{j\in C}d_j (j) \notag\\
  a(\ell) &= \bigl(\frac1{\delta_j}\deg\cO_\cX(Y_\ell)|_{Y_j}
            \bigr)_{j\in C} =
            \bigl(\frac1{p^{n_j}}(Y_j.Y_\ell)\bigr)_{j\in C}
   \qquad (\ell\in C)   \label{eq:raynaud-defs}
\\
b(m) &= \sum_{j\in C}\delta_jm_j,
   \qquad{m=(m_j)\in \ZZ^C} \notag
\end{align}
and $\comp{J}=\ker(b)/\im(a)$.  

If $\cX$ is semistable (meaning that $\cX_s$ is smooth over $k$ apart
from double points with distinct tangents), then both the character group and component
group can be described in terms of the reduced graph $\Gamma_{\cX_s}$.
The character group equals the homology of $\Gamma_{\cX_s}$. The
map $a\colon \ZZ[C] \to \ker(b)=\ZZ^{C,0}\subset \ZZ^C$ is, after 
identifying $\ZZ[C]$ with
$\ZZ^C$, the $0$-Laplacian $\laplace=\laplace_0$ 
(as in \cite{JRS22}) of the graph $\Gamma_{\cX_s}$,
which takes a vertex $v\in C$ to $\sum (v)- (v')\in \ZZ[C]\zero$,
the sum taken over all edges joining $v$ to an adjacent vertex $v'$.

In general, we have an isomorphism of $\Gal(\bar k/k)$-modules $\comp{J}
=\ker(b)/\im(a)$, where $a$, $b$ are the maps in the complex
\eqref{eq:raynaud} for the base change $\cX\otimes_R \Rsh$.

\subsection{Generalized Picard schemes of singular curves}
\label{sec:singular2}

Let $k$ be a field, and $Y/k$ a proper curve (in the sense of
\S\ref{sec:singular} above). Write $k'$ for the $k$-algebra
$\Gamma(Y,\cO_Y)$. By a \emph{generalized modulus} on $Y$ we mean a morphism of
$k$-schemes $\kModulus\to Y$, where $\kModulus$ is a finite
$k$-scheme, flat over $\Spec k'$.

\begin{lem}
  \label{lem:rigid}
  Let $g\colon \kModulus\to Y$ be a generalized modulus. Suppose that
  $g(\kModulus)$ meets each connected component of $Y$. Then
  $(\kModulus, g)$ is a rigidifier\footnote{``rigidificateur'' in
    \cite{Ray70}, ``rigidificator'' in \cite{BLR90}.} of $\Pic_{X/k}$, in
  the sense of \cite[(2.1.1)]{Ray70}
\end{lem}

\begin{proof}
  For $(\kModulus, g)$ to be a rigidifier, it is necessary and
  sufficient that for every $k$-algebra $A$, the map
  $g^* \colon \Gamma(Y \otimes_k A, \cO_{Y\otimes A}) \to
  \Gamma(\kModulus \otimes A, \cO_{\kModulus\otimes A})$ is injective. As
  $k$ is a field it is enough to show this for $A = k$, and this holds
  since by hypothesis $\kModulus/k'$ is faithfully flat.
\end{proof}

We define $\Pic_{(Y,\kModulus)/k}$ to be the scheme classifying line
bundles on $Y$ together with a trivialisation of the pullback to
$\kModulus$. Precisely, consider the functor $\cF$ which to a
$k$-scheme $S$ associates the set of equivalence classes of pairs
$(\cL,\alpha)$, where $\cL$ is a line bundle on $Y\times S$ and
$\alpha \colon \cO_{\kModulus\times S} \isom (g\times \id_S)^*\cL$ is a
trivialisation, and where pairs $(\cL,\alpha)$ , $(\cL',\alpha')$ are
equivalent if there exists an isomorphism
$\sigma\colon \cL \isom \cL'$ such that
$\alpha'=g^*(\sigma)\circ\alpha$.

Let $Y = Y_1 \sqcup Y_2$ where $g(\kModulus)$ is disjoint from $Y_2$ and
meets each connected component of $Y_1$. If $Y_2 = \emptyset$ then by
Lemma \ref{lem:rigid} we are in the situation of \cite[\S2]{Ray70}, and $\cF$ is a
sheaf for the fppf topology which we denote $\Pic_{(Y,\kModulus)/k}$. In
general, we define $\Pic_{(Y,\kModulus)/k}$ to be the sheafification of
$\cF$ for the fppf topology. Obviously
$\Pic_{(Y,\kModulus)/k} = \Pic_{(Y_1 ,\kModulus)/k} \times_k
\Pic_{Y_2/k}$. Put $k' = k_1 \times k_2$ where
$k_i = \Gamma(Y_i , \cO_{Y_i} )$. From \cite{Ray70} we then obtain:

\begin{prop}
  The functor $\Pic_{(Y,\kModulus)/k}$ is represented by a smooth
  $k$-group scheme, and there is an exact sequence of smooth group
  schemes
  \begin{equation}\label{eq:gen_pic_filtration}
    0 \to H \to \Pic_{(Y,\kModulus)/k} \to \Pic_{Y/k} \to 0
  \end{equation}
  where
  \[
    H = H_\Sigma \defeq \coker\left( \R_{k'/k}(\Gm) \to \R_{\kModulus/k}(\Gm)\right).
  \]
\end{prop}

\begin{proof}
  From (2.1.2), (2.4.1) and (2.4.3) of \cite{Ray70} we get the
  representability of $\Pic_{(Y_1 ,\kModulus)/k}$ along with an exact
  sequence\mpa
  \[
    0 \to \R_{k_1/k} \Gm \to \R_{\kModulus/k} \Gm \to
    \Pic_{(Y_1,\kModulus)/k} \to \Pic_{Y_1/k} \to 0
  \]
  of smooth group schemes (since, in this setting, Raynaud’s
  $\Gamma_X^*$ and $\Gamma_R^*$ are just $\R_{k_1/k} \Gm$ and
  $\R_{\kModulus/k}\Gm$). By \S1.1  above, the quotient
  $H$ is a smooth group scheme, and taking products with $\Pic_{Y_2/k}$ gives
  the result.
\end{proof}

Let $\Pic_{(Y,\kModulus)/k}^0$ denote the inverse image of
$\Pic_{Y/k}^0$ (classifying line bundles which are of degree zero on
every component of $Y$). Then $\Pic_{(Y,\kModulus)/k}^0$ is a smooth
connected $k$-group scheme of finite type.

\begin{exam}
  Suppose $Y$ is smooth over $k$ and absolutely
  irreducible, and that $\kModulus\to Y$ is a closed immersion. Then the image
  of $\kModulus$ is an effective divisor $\mathfrak{m}=\sum m_i(y_i)$
  for points $y_i\in Y(k)$. In this case $\Pic_{(Y,\kModulus)/k}^0$ is
  none other than the classical \cites{Ros54,Ser84} generalized Jacobian
  $J_{\mathfrak{m}}(Y)$ of $Y$.
  The isomorphism $J_\fm(Y) \isom \Pic^0_{(Y,\kModulus)/k}$ is given on
  $k$-points by mapping the class of a divisor
  $D\in\Div^0(Y\smallsetminus \kModulus)$ to the class of the pair
  $(\cO_Y(D),\alpha_{\mathrm{triv}})$, where $\alpha_{\mathrm{triv}}$ is
  the canonical trivialisation
  $\cO_{\kModulus} \isom \cO_Y|_\kModulus=\cO_Y(D)|_\kModulus$.
\end{exam}

Let $k$ be perfect. Then $\Pic_{(Y,\kModulus)/k}^0$ has a maximal
connected affine subgroup $\Pic_{(Y,\kModulus)/k}^{0,\lin}$ which is a
linear group, and its character group has the following combinatorial
description, generalizing \S\ref{sec:singular} above.

First suppose that $k$ is algebraically closed, and that $Y$,
$\kModulus$ are reduced. As in \S\ref{sec:singular}, let
$\phi\colon \Ytilde \to Y$ be the normalisation, and define
$A=Y^\sing$, $B=\phi^{-1}(A)$, $C=\pi_0(\Ytilde)$.
Decompose
$\kModulus=\kModulus^\sing\sqcup \kModulus^\reg$, where $z\in \kModulus^\sing$
(resp.~$\kModulus^\reg$) if $g(z)$ is a singular (resp.~smooth) point of
$Y$. There are maps
\[
\begin{tikzcd}
  B \arrow[r, "\psi"]  \arrow[d, "{\phi}"]
  & C \arrow[r, leftarrow, "\theta"] & \kModulus^\reg \\
  A \arrow[r, leftarrow, "\lambda"] & \kModulus^\sing
\end{tikzcd}
\]
where $\phi$, $\psi$ are as before, $\lambda$ is the restriction of
$g$ to $\kModulus^\sing$, and $\theta(z)$ is the component of $\Ytilde$
containing $g(z)$.

Define the \emph{extended graph of $(Y,\kModulus)$} to be the directed
graph
$\Gtilde_{Y,\kModulus}$ obtained by adding to the graph $\Gtilde_Y$
\begin{itemize}
\item a single vertex $\vzero$
\item for each $z\in \kModulus^\sing$, an edge from $\vzero$ to the vertex
  $\lambda(z)\in A\subset V(\Gtilde_Y)$
\item for each $z\in \kModulus^\reg$, an edge from $\vzero$ to the vertex
  $\theta(z)\in C\subset V(\Gtilde_Y)$.
\end{itemize}
If $Y$ only has double points and $\kModulus=\kModulus^\reg$, then we may
likewise define the \emph{reduced graph} $\Gamma_{Y,\kModulus}$, which
is the undirected graph obtained by adding to $\Gamma_Y$ a single
vertex $\vzero$ and, for each $z\in \kModulus$, an edge joining $\vzero$ to
$\theta(z)\in C=V(\Gamma_Y)$. As before, the geometric realisations of
$\Gtilde_{Y,\kModulus}$ and $\Gamma_{Y,\kModulus}$ are canonically
homeomorphic.

For arbitrary perfect $k$ and proper curve $Y$, we define
$\Gtilde_{Y,\kModulus}$, $\Gamma_{Y,\kModulus}$ to be the graphs
attached to the curve with modulus
$(Y^\red\otimes\kbar, \kModulus^\red\otimes\kbar)$, which are graphs
with a continuous action of $\Gal(\kbar/k)$.

\begin{prop}\label{prop:charGenJac}\mbox{}
\begin{enumerate}[\upshape (a)]
\item
\label{prop:charGenJac1}
  The character group $\Hom(\Pic^{0,\lin}_{(Y,\kModulus)/k},\Gm)$ is canonically
  isomorphic to $H_1(\Gtilde_{Y,\kModulus},\ZZ)$, as $\Gal(\kbar/k)$-module.
\item
\label{prop:charGenJac2}
 If $Y^\red$ has only double points, then
  $\Hom(\Pic^{0,\lin}_{(Y,\kModulus)/k},\Gm)\simeq H_1(\Gamma_{Y,\kModulus},\ZZ)$.
\end{enumerate}
\end{prop}

\begin{proof}
  We may assume that $k$ is algebraically closed; the Galois
  equivariance of the isomorphisms will be clear from the
  construction. By the homeomorphism between the extended and reduced
  graphs, it suffices to prove \eqref{prop:charGenJac1}.
  
The map
$g^\red\colon \kModulus^\red \to Y^\red$ is a reduced modulus, and the
obvious morphism induced by pullback
\[
\Pic_{(Y,\kModulus)/k}^0 \to \Pic_{(Y^\red,\kModulus^\red)/k}^0
\]
has unipotent kernel, since the same is true for the maps
$\R_{\kModulus/k}\Gm\to \R_{\kModulus^\red/k}\Gm$ and $\Pic_{Y/k}^0 \to
\Pic_{Y^\red/k}^0$. So the character group of $\Pic_{(Y,\kModulus)/k}^{0,\lin}$
is unchanged by passing to reduced subschemes; hence we may assume that
both $Y$ and $\kModulus$ are reduced. Next, let $Y'\to Y$ be
the seminormalisation. Then as $\kModulus$ is reduced, $\kModulus \to Y$
factors uniquely through $Y'$, and the resulting map
\[
\Pic_{(Y,\kModulus)/k}^0 \to \Pic_{(Y',\kModulus)/k}^0
\]
has unipotent kernel. So we may assume in addition that $Y$ is
seminormal. Finally,  normalisation induces an exact sequence
\begin{equation}\label{eq:Gdef}
0 \to G \to  \Pic_{(Y,\kModulus)/k}^0 \xrightarrow{\phi^*} \Pic_{\tilde Y/k}^0 \to 0
\end{equation}
where $G$ classifies equivalence classes of pairs $(\cL,\beta)$, where
$\cL$ is a line bundle on $Y$ whose pullback to $\tilde Y$ is
trivial, and $\beta$ is a trivialisation of the pullback of $\cL$ to
$\kModulus$. There is a surjective map
\begin{equation}
  \label{eq:singular_tori}
  \Gm^B \times \Gm^{\kModulus^\sing} \times \Gm^{\kModulus^\reg} \to G
\end{equation}
given as follows: a tuple
\[
  ((a_x)_{x\in B},(b_z)_{z\in \kModulus^\sing},
  (c_z)_{z\in \kModulus^\reg})\in (\Gm^B \times \Gm^{\kModulus^\sing} \times
  \Gm^{\kModulus^\reg})(k)
\]
determines:
\begin{itemize}
\item[(i)] for every $y\in A$, and any $x$, $x'\in \phi^{-1}(y)$,
  isomorphisms $a_x^{-1}a_{x'} \colon x^*\cO_{\tilde Y}=k \isom
  k=x^{\prime*}\cO_{\tilde Y}$ satisfying the cocycle condition, and thus a
  descent of $\cO_{\tilde Y}$ to a line
  bundle $\cL$ on $Y$
\item[(ii)] for every $z\in \kModulus^\sing$, and every
  $x\in \phi^{-1}(g(z))$, a trivialisation
  $b_za_x^{-1}\colon k\isom k=x^*\cO_{\tilde Y}$.
  These trivialisations are compatible with the descent data (i) and
  therefore give trivialisations $k\isom z^*\cL$ for every
  $z\in \kModulus^\sing$.
\item[(iii)] for every $z\in \kModulus^\reg$, a trivialisation
  $k\isom z^*\cL=k$ given by multiplication by
  $c_z$.
\end{itemize}

What is the kernel of the map \eqref{eq:singular_tori}?
Fix $y\in A$. Then multiplying
$a_x$, for $x\in\phi^{-1}(y)$, and $b_z$, for $z \in \kModulus^\sing$ such that
$g(z)=y$, by a common element of $k^\mul$ does not change the descent data (i)
or the trivialisation (ii), so we obtain the same $(\cL,\beta)$. The
equivalence relation on pairs is realised by the automorphism group
$\Gm^C$ of $\phi^*\cL=\cO_{\tilde Y}$, which acts on tuples by
\[
(d_j)_{j\in C} \colon
((a_x)_{x\in B},(b_z)_{z\in \kModulus^\sing}, (c_x)_{x\in \kModulus^\reg})
\mapsto
((d_{\psi(x)}a_x), (b_z), (d_{\theta(z)}c_z)).
\]
Therefore $G$ is the torus whose character group is the kernel of the
map
\begin{equation}
  \label{eq:pizza1}
  \ZZ[B] \oplus \ZZ[\kModulus^\sing] \oplus \ZZ[\kModulus^\reg] \to \ZZ[C] \oplus \ZZ[A] 
\end{equation}
with matrix
\[
\begin{bmatrix}
  \psi & 0 & \theta \\
  \phi & \lambda & 0
\end{bmatrix}\,.
\]
The homology complex of $\Gtilde_{Y,\kModulus}$ is
\begin{equation}
  \label{eq:pizza2}
  \ZZ[B] \oplus \ZZ[\kModulus^\sing] \oplus \ZZ[\kModulus^\reg] \to \ZZ[C] \oplus \ZZ[A]
  \oplus \ZZ
\end{equation}
with differential given by the matrix
\[
\begin{bmatrix}
  \psi & 0 & \theta \\
  -\phi & \lambda & 0 \\
  0 & -\varepsilon & -\varepsilon
\end{bmatrix}
\]
where $\varepsilon \colon \ZZ[\kModulus^?] \to \ZZ$ is the augmentation
$z\mapsto 1$, for $z\in\kModulus^?$, $?\in\{\reg,\sing\}$. There is an obvious map from the complex
\eqref{eq:pizza2} to the complex \eqref{eq:pizza1} which induces an
isomorphism on kernels. Since $G$ is by \eqref{eq:Gdef} the maximal
multiplicative quotient of $\Pic^{0,\lin}_{(Y,\kModulus)/k}$, this
gives the isomorphism \eqref{prop:charGenJac1}. The construction is Galois equivariant
by transport of structure.
\end{proof}

\subsection{Functoriality I}
\label{sec:functoriality1}

Let $g\colon \Sigma \to Y$, $g'\colon \Sigma' \to Y'$ be generalized
moduli on proper curves over $k$ as in the previous section, and suppose we
have finite morphisms $f$, $f_\Sigma$ fitting into a commutative
diagram
\begin{equation}
  \label{eq:funct1}
  \begin{tikzcd}
    \Sigma' \arrow[r, "g'"] \arrow[d, "{f_\Sigma}"]
    & Y' \arrow[d, "f"]  \\
    \Sigma \arrow[r, "g"] & Y
  \end{tikzcd}
\end{equation}
Then there is an associated pullback morphism
\[
  (f,f_\Sigma)^* \colon \Pic_{(Y,\Sigma)/k} \to \Pic_{(Y',\Sigma')/k}
\]
taking a pair $(\cL, \alpha\colon \cO_\Sigma \isom g^*\cL)$ to the
pair
\[
  \bigl(f^*\cL,\ f_\Sigma^*\alpha \colon \cO_{\Sigma'} \isom
  f_\Sigma^*g^*\cL=g^{\prime*}(f^*\cL)\bigr),
\]
which we will simply denote by $f^*$ if no confusion can arise.

To define pushforward, consider the
commutative diagram
\[
  \begin{tikzcd}
    \Sigma' \arrow[r,"h"] \arrow[dr, "f_\Sigma"'] &
    \Sigma\times_Y Y' \arrow[r, "\pr_2"] \arrow[d, "\pr_1"] &
    Y' \arrow[d, "f"] \\
    & \Sigma \arrow[r, "g"] & Y
  \end{tikzcd}
\]
We assume that $f$ is flat, and that $h$ is a closed immersion whose
ideal sheaf $\cI$ is nilpotent and satisfies
\begin{equation}
  \label{eq:pushforward-hyp}
  \norm_{\Sigma\times_YY'/\Sigma}(1+\cI) = \{1\}.
\end{equation}
We may then define a morphism
$f_* =(f,\Sigma)_* \colon \Pic_{(Y',\Sigma')/k} \to \Pic_{(Y,\Sigma)/k}$ by
$f_*\colon (\cL',\alpha') \mapsto (\cL,\alpha)$, where
$\cL=\norm_f(\cL')$, the norm of $\cL'$ \cite[6.5]{EGA2} and $\alpha$
is given as follows: if $\cI=0$ is zero, then \eqref{eq:funct1} is
Cartesian, and $\alpha$ is the composite
\[
  \alpha \colon \cO_\Sigma \xrightarrow[\norm_f(\alpha')]{\sim}
  \norm_f(g^{\prime*}\cL')
  \isom g^*\cL
\]
(the second isomorphism given by \cite[(6.5.8)]{EGA2}). In general, 
$\alpha'\colon \cO_{\Sigma'} \isom g^{\prime*}\cL'$ can at least locally be
extended to an isomorphism $\alpha'' \colon \cO_{\Sigma\times_YY'} \isom
\pr_2^*\cL'$, well-defined up to local sections of $1+\cI$. Taking
norms, we then get a well-defined global isomorphism  
$\alpha=\norm_{\Sigma\times_Y Y'/\Sigma}(\alpha'')\colon \cO_\Sigma
\isom g^*\cL$.

The maps $f^*$, $f_*$ preserve $\Pic^0$ in
all cases.

\begin{exam} \label{ex:functoriality1}
Suppose that $Y$, $Y'$ are smooth over $k$ and
absolutely irreducible, and that
$\Sigma\subset Y$, $\Sigma'\subset Y'$ are closed subschemes defined
by reduced moduli $\fm$, $\fm'$. Let $J_\fm=\Pic^0_{(Y,\Sigma)/k}$,
$J'_{\fm'}=\Pic^0_{(Y',\Sigma')/k}$,  be the associated generalized
Jacobians. Let $f\colon Y' \to Y$ be a finite morphism with $f^{-1}(\Sigma)^{\red}=\Sigma'$. Then \eqref{eq:pushforward-hyp} holds, and therefore we get morphisms
\[
f^*\colon J_{\fm} \to J'_{\fm'},\quad
f_*\colon J'_{\fm'} \to J_{\fm}.
\]
\end{exam}
If $f'\colon Y' \to Y$ is another finite morphism with
$f^{\prime-1}(\kModulus)\supset\kModulus'$, then we get an induced
endomorphism $f_*f^{\prime*}\colon J_\fm \to J_\fm$, compatible with the
usual correspondence action on $J$ (pullback along $f'$ followed by
norm with respect to $f$). For later reference, we will say
that the modulus $\fm$ is \emph{stable} under the correspondence
$f_*f^{\prime*}$. 

Returning to the general case, assume that $k$ is algebraically
closed, that $f$ is flat and that \eqref{eq:pushforward-hyp} holds.  Write
\[
\XX =\Hom(\Pic^{0,\lin}_{(Y,\Sigma)/k},\Gm),\quad
\XX' =\Hom(\Pic^{0,\lin}_{(Y',\Sigma')/k},\Gm)
\]
for the character groups of the linear parts of the generalized Picard schemes.
Then $f^*$, $f_*$ induce by functoriality homomorphisms
\begin{equation}
\label{eq:funct2}
\XX(f^*) \colon \XX' \to \XX, \quad \XX(f_*) \colon \XX \to \XX'.
\end{equation}
By Proposition \ref{prop:charGenJac} and \eqref{eq:pizza1}
 we have canonical isomorphisms
 \begin{align*}
   \XX &=\ker\biggl( \begin{bmatrix}  \psi & 0 & \theta \\
  \phi & \lambda & 0\end{bmatrix} \colon
  \ZZ[B] \oplus \ZZ[\kModulus^\sing] \oplus \ZZ[\kModulus^\reg] \to
                   \ZZ[C] \oplus \ZZ[A]
                   \biggr)\\
   &= H_1(\Gtilde_{Y,\Sigma},\ZZ)
\end{align*}
where $A=(Y^\red)^{\sing}$, $B=\phi^{-1}(A)$, $C=\pi_0(\Ytilde)$, and
similarly for $\XX'$. We now describe the maps \eqref{eq:funct2}
combinatorially, under further hypotheses. Let $A'$, $B'$, $C'$ denote
the corresponding sets for $Y'$, and assume the following.
\begin{hyps}\mbox{ }
\label{hyp:funct-simp}
\begin{enumerate}[(i)]
\item
\label{hyp:funct-simp1}
$f^{-1}(A)=A'$, and $f$ is \'etale at each point of $A'$.
\item
\label{hyp:funct-simp2}
$\Sigma^\sing=\emptyset =\Sigma^{\prime\sing}$, and
  $\Sigma$, $\Sigma'$  are reduced.
\item
\label{hyp:funct-simp3}
$\Sigma'=f^{-1}(\Sigma)^\red$.
\end{enumerate}
\end{hyps}
Hypotheses \ref{hyp:funct-simp}\eqref{hyp:funct-simp2} and
\eqref{hyp:funct-simp3} together imply that \eqref{eq:pushforward-hyp}
is satisfied.  Then $f$ induces maps $A'\to A$, $B'\to B$, $C'\to C$
which we also denote by $f$. The diagram
\begin{equation}
  \label{eq:funct3}
  \begin{tikzcd}
    A' \arrow[r, leftarrow, "\phi'"] \arrow[d, "f"]
    & B' \arrow[r, "\psi'"] \arrow[d, "f"]
    & C' \arrow[r, leftarrow, "\theta'"]\arrow[d, "f"]
    & \Sigma' \arrow[d, "f"]   \\
    A \arrow[r, leftarrow, "\phi"] & B \arrow[r, "\psi"] & C \arrow[r,
    leftarrow, "\theta"] & \Sigma 
  \end{tikzcd}
\end{equation}
commutes, and so we have a commutative square
\begin{equation}
  \label{eq:funct4}
  \begin{tikzcd}[ampersand replacement=\&, column sep=large]
    \ZZ[B']\oplus \ZZ[\Sigma']
    \arrow[r, "{\left[\begin{smallmatrix}\psi'&\theta'\\\phi'&0\end{smallmatrix}\right]}"]
    \arrow[d, "f"] \&
    \ZZ[C']\oplus \ZZ[A'] \arrow[d, "f"]  \\
    \ZZ[B]\oplus \ZZ[\Sigma]
    \arrow[r, "{\left[\begin{smallmatrix}\psi&\theta\\\phi&0\end{smallmatrix}\right]}"] 
\&     \ZZ[C]\oplus \ZZ[A] 
  \end{tikzcd}
\end{equation}

\begin{prop} \label{prop:upper-star}
  Assume Hypotheses \textup{\ref{hyp:funct-simp}}.
  The homomorphism $\XX(f^*)$ is induced by the vertical maps in \eqref{eq:funct4}.
\end{prop}

\begin{proof}
  We first observe that we may assume in addition that $Y$ and $Y'$
  are seminormal (and therefore reduced). Indeed, the descriptions of the
  character groups $\XX$ and $\XX'$ is unchanged after replacing the
  curves by their seminormalisations. It remains to verify that the
  induced map on seminormalisations $f^{\sn} \colon Y^{\prime\sn} \to
  Y^{\sn}$ is flat. But $Y^{\sn}\setminus A = Y^{\red}\setminus A$ is
  smooth, and so the restriction of $f^{\sn}$ to
  $Y^{\prime\sn}\setminus A'$ is automatically flat. By hypothesis,
  there is a neighbourhood $U\subset Y$ of $A$ such that $f \colon
  U'\defeq f^{-1}(U) \to U$ is \'etale. Then by \cite[Prop. 5.1]{GT80},
  we have a Cartesian square
  \[
    \begin{tikzcd}
      U^{\prime\sn} \arrow[r] \arrow[d, "{f^{\sn}}"]
      & U' \arrow[d, "{f}"]  \\
      U^{\sn} \arrow[r] & U
    \end{tikzcd}
  \]
  and in particular $f^{\sn}$ restricted to $U^{\prime\sn}$ is
  \'etale.

  We now compute the dual map $f^* \colon \Pic^{0,\lin}_{(Y,\Sigma)/k}
  \to \Pic^{0,\lin}_{(Y',\Sigma')/k}$ which is a morphism of tori
  (since we are assuming that $Y$ and $Y'$ are seminormal and
  $\Sigma$, $\Sigma'$ are reduced). As explained in
  \S\ref{sec:singular2}, a $k$-point of $\Pic^{0,\lin}_{(Y,\Sigma)/k}$
  is represented by a pair $((a_x)_{x\in B},(c_z)_{z\in \Sigma}) \in
  (k^\mul)^B\times (k^\mul)^\Sigma$, where $(a_x)$ determines
  descent data for $\cO_{\Ytilde}$ with respect to the normalisation
  morphism $\phi\colon\Ytilde \to Y$, and $(c_z)$ determines
  trivialisations $\times c_z\colon k\xrightarrow{\sim}k=z^*\cO_{\Ytilde}$ which
  descend to a rigidification along $\Sigma$ of the descended line
  bundle.

  Let $a'_{x'}=a_{f(x')}$ ($x'\in B'$) and $c'_{z'}=c_{f(z')}$ ($z'\in
  \Sigma'$). Then if $(\cL,\alpha) \in
  \Pic^{0,\lin}_{(Y,\Sigma)/k}(k)$ is represented by the pair $((a_x),
  (c_z))$, the pullback $f^*(\cL,\alpha)$ is represented by
  $((a'_{x'}),(c'_{z'}))$. The obvious map
  \[
    \Aut\cO_{\Ytilde}=(k^\mul)^C  \to
    \Aut\cO_{\Ytilde'}=(k^\mul)^{C'}
  \]
  is induced by $f\colon C'\to C$,
  and therefore $\XX(f^*)$ is induced by the vertical maps $f$ in
  \eqref{eq:funct4} as required.  
\end{proof}

We now compute $\XX(f_*)$.  Let
\[
  f^*\colon
  \begin{cases}
    \ZZ[\Sigma] \to \ZZ[\Sigma']&\\[0.2ex]
    \ZZ[A] \to \ZZ[A']&\\[0.2ex]
    \ZZ[B] \to \ZZ[B']&
  \end{cases}
\]
be the inverse image maps on divisors. By Hypothesis
\ref{hyp:funct-simp}(\ref{hyp:funct-simp1}), this means 
that if $x\in A$ or $x\in B$, then $f^*\colon (x) \mapsto
\sum_{f(x')=x}(x')$, and if $z\in\Sigma$, then
\[
  f^*\colon (z) \mapsto \sum_{f(z')=z} r_{z'/z}(z')
\]
where $r_{z'/z}$is the ramification degree of $f$ at $z'$. Finally,
define $f^*\colon \ZZ[C] \to \ZZ[C']$ by
\[
  f^*\colon (Z) \mapsto \sum_{f(Z')=Z} [\kappa(Z'):\kappa(Z)]\,(Z')
\]
where $Z\subset \Ytilde$, $Z'\subset \Ytilde'$ are connected components.
These maps fit into the  diagram
\begin{equation}
  \label{eq:funct5}
  \begin{tikzcd}[ampersand replacement=\&, column sep=large]
    \ZZ[B]\oplus \ZZ[\Sigma]
    \arrow[r, "{\left[\begin{smallmatrix}\psi&\theta\\\phi&0\end{smallmatrix}\right]}"] \arrow[d, "f^*"]  \&
    \ZZ[C]\oplus \ZZ[A] \arrow[d, "f^*"]  \\
    \ZZ[B']\oplus \ZZ[\Sigma']
    \arrow[r, "{\left[\begin{smallmatrix}\psi'&\theta'\\\phi'&0\end{smallmatrix}\right]}"] \&
    \ZZ[C']\oplus \ZZ[A'] \,.
  \end{tikzcd}
\end{equation}

\begin{prop} \label{prop:lower-star}
  Assume Hypotheses \textup{\ref{hyp:funct-simp}}.
  The diagram \eqref{eq:funct5} is commutative, and the vertical maps
  induce the homomorphism $\XX(f_*)\colon \XX \to \XX'$.
\end{prop}

\begin{proof}
As in \ref{prop:upper-star}, we may assume that $Y$ and $Y'$ are
seminormal. Consider again the dual map of tori $f^*\colon \Pic^{0,\lin}_{(Y',\Sigma')/k}
  \to \Pic^{0,\lin}_{(Y,\Sigma)/k}$.  Let $(\cL',\alpha')\in
  \Pic^{0,\lin}_{(Y',\Sigma')/k}(k)$, represented by the pair
  $((a'_{x'})_{x'\in B'}, (c'_{z'})_{z'\in \Sigma'})$. Since $f$ is
  \'etale at $B'$, the normalised map $\tilde f \colon
  \Ytilde'\to\Ytilde$ induces a norm homomorphism
  \[
    \norm_{\tilde f} \colon
    \Gamma(\phi^{\prime-1}(Y^{\prime\sing}),\cO^\mul) = (k^\mul)^{B'}
    \to \Gamma(\phi^{-1}(Y^{\sing}),\cO^\mul) = (k^\mul)^{B}
  \]
  which equals the homomorphism $f_! \colon (k^\mul)^{B'} \to
  (k^\mul)^{B}$ given by
  \[
    f_!\colon (a'_{x'})_{x'\in B'} \mapsto (a_x)_{x\in B},\quad a_x=
    \prod_{f(x')=x} a'_{x'}.
  \]
  The analogous statement holds for
  \[
    \norm_{f} \colon
    \Gamma(Y^{\prime\sing},\cO^\mul) = (k^\mul)^{A'}
    \to \Gamma(Y^{\sing},\cO^\mul) = (k^\mul)^{A}\,.
  \]
  Since $f$ is \'etale at $A'$, the square
  \[
    \begin{tikzcd}
      B' \arrow[r, "{\phi'}"] \arrow[d, "f"]
      & A' \arrow[d, "f"]  \\
      B \arrow[r, "\phi "] & A
    \end{tikzcd}
  \]
  is in fact Cartesian, and therefore
  \[
    \phi^*\circ f_! = f_! \circ \phi^{\prime*} \colon (k^\mul)^{B'}
    \to (k^\mul)^A\,.
  \]
  Next, we consider the rigidification $\alpha'\colon \cO_{\Sigma'}
  \isom g^{\prime*}\cL=g^{\prime*}\cO_{\Ytilde}=\cO_{\Sigma'}$ given
  by multiplication by $(c'_{z'}) \in (k^\mul)^{\Sigma'}$. Let
  $\Sigma'' =f^{-1}(\Sigma)$ be the scheme-theoretic inverse image of
  $\Sigma$. So $\Sigma''=\coprod_{z'\in\Sigma'} \tilde z'$ say, where
  $\tilde z'\simeq\Spec k[t]/(t^{r_{z'/z}})$. According to (ii) above,
  to compute $f_*(\cL',\alpha')$ we need to extend $\alpha'$ to a
  rigidification
  \[
    \alpha'' \colon \cO_{\Sigma''}  \isom
    \cL'|_{\Sigma''}=\cO_{\Sigma''}
  \]
  and we may as well take $\alpha''$ to be the sum of the maps
  $\cO_{\tilde z'} \isom \cO_{\tilde z'}$ given by multiplication by
  $c'_{z'}$. Then $\norm(\alpha'')\colon \cO_\Sigma \isom \cO_\Sigma$
  is multiplication by $(c_z)=\hat f_!(c'_{z'})$, where $\hat
  f_!\colon (k^\mul)^{\Sigma'} \to (k^\mul)^\Sigma$ is the map
  \[
    \hat f_! \colon (c'_{z'}) \mapsto (c_z),\quad c_z=\prod_{f(z')=z}
    (c'_{z'})^{r_{z'/z}}
  \]
  whose dual is the map $f^*\colon \ZZ[\Sigma'] \to \ZZ[\Sigma]$
  defined above.

  Finally we need to compute the action of
  $\Aut \cO_{\Ytilde'}=(k^\mul)^{C'}$. From \S\ref{sec:singular2} we
  know that $d'\in (k^\mul)^{C'}$ maps
  $((a'_{x'})_{x'\in B'}, (c'_{z'})_{z'\in\Sigma'})$ to
  $((d'_{\psi'(x')}a'_{x'}), (d'_{\theta'(z')}c'_{z'}))$, which under
  the norm maps to
  \begin{equation}
    \label{eq:funct6}
    \biggl( \Bigl(\prod_{\tilde f(x')=x} d'_{\psi'(x')}a'_{x'}\Bigr)_{x\in B},
      \Bigl(\prod_{\tilde f(z')=z}
      (d'_{\theta'(z')}c'_{z'})^{r_{z'/z}}\Bigr)_{z\in \Sigma}
    \biggr)\,.
  \end{equation}
  Let $x\in B$ be fixed. Then if $Z=\psi(x)\in C$ is the component
  containing $x$, and $Z'\in C'$ is a component of $\Ytilde'$ lying
  over $Z$, the set $f^{-1}(x)\cap Z'$ has cardinality
  $[\kappa(Z'):\kappa(Z)]$, since $f$ is \'etale at
  $f^{-1}(x)$. Therefore
  \[
    \prod_{\tilde f(x')=x} d'_{\psi'(x')} =
    \prod_{\substack{Z'\in C'\\ f(Z')=\psi(x)}}
    (d'_{Z'})^{[\kappa(Z'):\kappa(Z)]}\,.
  \]
  Similarly, let $z\in \Sigma$ be fixed, and $Z=\theta(z)\in C$ the
  component of $\Ytilde$ containing it. Then if $Z'\in C'$ is a
  component of $\Ytilde'$ lying over $Z$,
  \[
    \sum_{z'\in f^{-1}(z)\cap Z'} r_{z'/z} = [\kappa(Z'):\kappa(Z)]
  \]
  and therefore
  \[
    \prod_{\tilde f(z')=z} (d'_{\theta'(z')})^{r_{z'/z}}
    = \prod_{\substack{Z'\in C'\\ f(Z')=\theta(z)}}
    (d'_{Z'})^{[\kappa(Z'):\kappa(Z)]}\,.
  \]
  In other words, the pair \eqref{eq:funct6} equals
  \[
    (d_{\psi(x)}(f_!a')_x, d_{\theta(Z)}(\hat f_!c')_Z)
  \]
  where
  \[
    d_Z = \prod_{\substack{Z'\in C'\\ f(Z')=Z}}
    (d'_{Z'})^{[\kappa(Z'):\kappa(Z)]}\,.
  \]
  The dual of this map $d'\mapsto d$ is therefore the homomorphism
  $f^*\colon \ZZ[C] \to \ZZ[C']$ defined above.
\end{proof}

\subsection{Generalized Jacobians over DVRs}
\label{sec:hard-bit}

We resume the notations and hypotheses of \S\ref{sec:raynaud}. Let
$(x_i)_{i\in I}$ be a nonempty finite family of distinct closed points
of $X$, whose residue fields $\F_i$ are separable over $\F$. Let
$\fm=\sum_{i\in I}(x_i)$ be the associated  modulus on $X$, and
$J_\fm=\Pic^0_{(X,\fm)/\F}$ the generalized Jacobian of $X$ with
respect to $\fm$. The semiabelian variety $J_\fm$ is an extension of
$J$ by the torus
\[
  T_\fm=\Bigl(\prod_{i\in I} \R_{\F_i/\F}\Gm\Bigr)/\Gm.
\]
Write $\cJ_\fm$ for the N\'eron model of $J_\fm$.

Let $R_i$ be the integral closure of $R$ in $\F_i$. Then the inclusion
of the points $(x_i)$ in $X$ extends to a unique morphism
\[
 \Sigma\defeq \coprod_{i\in I}\Spec R_i \xrightarrow{g} \cX.
\]
As $\Gamma(\cX_s, \cO_{\cX_s} ) = k$, the special fibre
$g_s \colon \Sigma_s \to \cX_s$ is a generalized modulus, in the sense
of the previous section.  By Proposition \ref{prop:tori-neron}(b) the N\'eron model $\cT_\fm$ of $T_\fm$ equals
$(\R_{\Sigma/S} \cGm )/\cGm$, and its identity subgroup is
$\cT_\fm^0 = (\R_{\Sigma/S}\Gm )/\Gm$. 
\begin{lem}
  The pair $(\Sigma, g)$ is a rigidifier \textup{\cite[(2.1.1)]{Ray70}} of
  $\Pic_{\cX/S}$.
\end{lem}

\begin{proof}
  Let $S'$ be any $S$-scheme. Since $\cX/S$ is cohomologically flat
  and $\Sigma$ is flat over $S$, we have
  \begin{gather*}
    \Gamma(X \times_S S', \cO_{\cX\times_S S'}) = \Gamma(S', \cO_{S'})\quad\text{and}
  \\
  \Gamma(\Sigma \times_S S',\cO_{\Sigma\times_S S'}) = \Gamma(\Sigma, \cO_\Sigma) \otimes_R
  \Gamma(S' , \cO_{S'}).
\end{gather*}
  As $\Sigma$ is nonempty,
  $R \to \Gamma(\Sigma, \cO_\Sigma)$ is a split injection of
  $R$-modules, and therefore
  $\Gamma(\cX \times_S S', \cO_{\cX\times_S S'}) \to \Gamma(\Sigma
  \times_S S', \cO_{\Sigma\times_S S'})$ is injective.
\end{proof}

Let $P_\Sigma$ denote the rigidified Picard functor of
\cite[(2.1)]{Ray70}: for any $S$-scheme $S'$, $P_\Sigma(S')$ is the group
of equivalence classes of pairs $(\cL, \alpha)$, where $\cL$ is a line
bundle on $\cX \times_S S'$, and
$\alpha \colon \cO_{\Sigma\times_SS'}\isom (g \times \id_{S'})^* \cL$
is a trivialisation. Pairs $(\cL, \alpha)$ and $(\cL', \alpha')$ are
equivalent if there exists an isomorphism
$\sigma\colon \cL \isom \cL'$ such that
$\alpha' = (g \times \id_{S'})^*(\sigma) \circ \alpha$. By \cite[(2.3.1--2)]{Ray70},
$P_\Sigma$ is a smooth algebraic space in groups over $S$, and we have
an exact sequence of algebraic spaces in groups \cite[(2.4.1)]{Ray70}
\begin{equation*}
  0 \to \cT_\fm^0 = \R_{\Sigma/S} \Gm /\Gm \to P_\Sigma
  \xrightarrow{r} P \to 0
\end{equation*}
where $r$ is the ``forget the rigidification'' functor. (Since $\cX/S$
is cohomologically flat and $f_*\cO_\cX=\cO_S$, one has
$\Gamma_X^*=\Gm$.)  If $S$ is strictly Henselian, $P_\Sigma$ is a
scheme; indeed, $P$ is a scheme, and $\cT_\fm^0$ is affine, so by flat
descent for affine schemes \cite[tag 0245]{Stacks}, the
$\cT_\fm^0$-torsor $P_\Sigma$ over $P$ is representable.

Define the sheaf $P_\fm$ to be the pushout of fppf sheaves:
\begin{equation}
\label{eq:PSigma}
  \begin{tikzcd}
    0 \arrow[r] & \cT_\fm^0 \arrow[r] \arrow[d, hook]
    & P_\Sigma\arrow[r, "r"] \arrow[d, hook]
    & P \arrow[r] \arrow[d, equal]& 0
    \\
    0 \arrow[r] & \cT_\fm \arrow[r] & P_\fm \arrow[r, "r'"] & P \arrow[r] & 0 
  \end{tikzcd}
\end{equation}
Explicitly, $P_\fm$ is the sheafification of the functor on
$S$-schemes
\begin{equation}
  \label{eq:Pfm-def}
  S' \mapsto \cT_\fm^0(S') \backslash (P_\Sigma(S')\times
  \cT_\fm(S'))
\end{equation}
where $\cT_\fm^0(S')$ acts on the product by $a(b, c) = (ab, a^{-1}c)$.

\begin{prop}
  \label{prop:Pm-smooth}
$P_\fm$ is a smooth algebraic space in groups over $S$. If $S$ is
strictly Henselian, $P_\fm$ is represented by a smooth $S$-group scheme.
\end{prop}

\begin{proof}
  We have an exact sequence
  $0 \to \cT_\fm^0 \to \cT_\fm \xrightarrow{\pi} s_* \comp{T} \to 0$ of
  $S$-group schemes. For $h \in \comp{T} = (s_*\comp{T})(S)$, let
  $U_h = \pi^{-1}(h)$, an affine open subscheme of $\cT_\fm$. Then
  $\cT_\fm$ is the union of the $U_h$, glued along their generic
  fibres. If $\hat h \in \cT_\fm(S) = T(\F)$ is any lift of $h$, then
  $U_h$ is the translate of $\cT_\fm^0$ by $\hat h$.  Therefore
  $P_\fm$ is the union of copies of $P_\Sigma$ indexed by $\comp{T}$,
  glued along their generic fibres by the isomorphism given by
  translation by $\hat h \in P_\Sigma(\F)$, and the result follows
  from the corresponding statement for $P_\Sigma$.
\end{proof}

This result implies that $P_\fm$ is determined by its restriction to
$(Sm/S)$, the category of essentially smooth $S$-schemes. We can
describe this functor explicitly. Let $\cF^*$ be the functor on
$(Sm/S)$ whose value on $S'$ is the group of equivalence classes of
pairs $(\cL, \beta = (\beta_i )_{i\in I})$, where $\cL$ is a line bundle
on $\cX \times_S S'$ and for each $i \in I$,
$\beta_i \colon \cO_{S'} \otimes_R \F_i \isom (x_i \times
\id_{S'})^*\cL$ is a trivialisation of $\cL$ at $x_i \times_S
S'$. Pairs $(\cL, \beta)$ and $(\cL', \beta')$ are equivalent if there
exists an isomorphism $\sigma \colon \cL \isom \cL'$ and some
$u \in \cO^\mul(S'\otimes_R \F)$ such that for every $i$ the diagram
\begin{equation}
  \label{eq:Fstar-equivalence}
  \begin{tikzcd}
    \cO_{S'\otimes_R\F_i} \arrow[r, "{\beta_i}"] \arrow[d, "{\times u}"]
    & (x_i\times \id_{S'})^*\cL \arrow[d, "{\sigma}"] 
    \\
    \cO_{S'\otimes_R\F_i} \arrow[r, "{\beta_i'}"] & (x_i\times \id_{S'})^*\cL'    
  \end{tikzcd}
\end{equation}
commutes. 
Note that $u$ is uniquely determined by $\sigma$. If $S' \in (Sm/S)$
is actually an $\F$-scheme, then giving a pair $(u, \sigma)$ is the
same as giving an isomorphism $(\cL, \beta) \isom (\cL', \beta')$,
since we can absorb $u$ into $\sigma$, and therefore the restrictions
of $\cF^*$ and $P_\Sigma$ to $(Sm/\F)$ are equal.

\begin{thm}
  \label{thm:Pm_eq_Fstar}
  The restriction of $P_\fm$ to $(Sm/S)_\et$ is the sheafification
  for the \etale\ topology of the presheaf $\cF^*$.
\end{thm}

\begin{proof} We have an exact sequence of fppf sheaves
\[
0 \to \Gm \xrightarrow{\diag} \R_{\Sigma/S} \Gm \times  \cGm 
\xrightarrow{\psi} P_\Sigma \times  \R_{\Sigma/S} \cGm
\]
where the map $\psi$ on $S'$-valued points is given by
\[
\psi \colon (a, b) \mapsto ((\cO_{\cX\times_S S'},a\cdot
\id_{\cO_{\Sigma\times S'}}), a^{-1} b) \in 
P_\Sigma (S') \times  \cGm (\Sigma \times_S S' ). 
\]
By definition, $P_\fm$ is the cokernel of $\psi$ in the category of
fppf sheaves. As the coimage of $\psi$ is a smooth $S$-group scheme,
$P_\fm$ is also the cokernel of $\psi$ in the category of \etale{}
sheaves. Let $S' \in (Sm/S)$ and consider the map
\[
\phi_{S'} \colon P_\Sigma (S' ) \times  \R_{\Sigma/S} \cGm (S' ) = P_\Sigma (S' )
\times  \Gm (\Sigma_\F \times_S S' ) \to \cF^* (S' ) 
\]
given as follows: let $(\cL, \alpha)$ represent an element of
$P_\Sigma (S')$ and $v \in \Gm (\Sigma_F \times_S S' )$. We map the
pair $((\cL, \alpha), v)$ to the equivalence class of $(\cL, \beta)$,
where
$\beta = \alpha \otimes v \colon \cO_{\Sigma\times S' \otimes \F}\isom
(g \times \id_{S' \otimes \F} )^* \cL$. It is easy to see that this is
well-defined and functorial, and that the resulting sequence of
presheaves on $(Sm/S)$
\[
\R_{\Sigma/S}\Gm \times  \cGm \xrightarrow{\psi}
 P_\Sigma \times  \R_{\Sigma/S} \cGm \xrightarrow{\phi}
\cF^*
\]
is exact. Moreover, for any $(\cL, \beta) \in \cF^* (S' )$, there
exists a Zariski cover $S'' \to S'$ such that $(\cL, \beta)|_{S''}$ is
in the image of $\phi_{S''}$. The result follows.
\end{proof}

Let $E_\fm$ denote the closure in $P_\fm$ of the zero section. It is
contained in 
\[
P_\fm' = 
\ker(\deg \colon P_\fm \to P \to \ZZ). 
\]

\begin{thm}\mbox{ }
  \label{thm:main}
\begin{enumerate}[\upshape (a)]
\item
\label{thm:main1}
The map $r'$ \eqref{eq:PSigma} induces an isomorphism $E_\fm \isom E$.
\item
\label{thm:main2}
The quotient $P_\fm' /E_\fm$ is represented by the N\'eron model $\cJ_\fm$ of $J_\fm$.
\item
\label{thm:main3}
There is an exact sequence of N\'eron models
\[
0 \to \cT_\fm \to \cJ_\fm \to \cJ \to 0.
\]
\item
\label{thm:main4}
Assume that S is strictly Henselian. Then there is a canonical
isomorphism 
\[
P_{\fm,s} /P_{\fm,s}^0 \isom
\ZZ^C \oplus \ZZ^I /e\ZZ
\]
where $e = (e_i) \colon \ZZ \to \ZZ^I$ is as in Proposition
\textup{\ref{prop:tori-neron}}.
\end{enumerate}
\end{thm}
The analogue of \eqref{thm:main1} need not hold for $P_\Sigma$ --- see 
Example \ref{ex:conic-picard} after the proof.

\begin{proof} (\ref{thm:main1}) By \cite[(3.3.5)]{Ray70} and Proposition
  \ref{prop:Pm-smooth}, $E_\fm$ is an \etale\ algebraic space in groups
  over $S$. So we may compute it by restriction to $(Sm/S)_\et$, using
  the description of Theorem \ref{thm:Pm_eq_Fstar}, and we may also
  assume that $S$ is strictly Henselian. In this case, from
  \S\ref{sec:raynaud} we have that $E(S)$ is generated by the classes
  of the line bundles $\cO_\cX (Y_j )$. Let
  $\beta_{\mathrm{triv}} = (\beta_{\mathrm{triv},i} )$ be the trivial
  rigidification of the generic fibre $\cO_\cX (Y_j )_\F =
  \cO_X$ at $(x_i)$. Then $E_\fm$ is generated by the equivalence classes of
  pairs $(\cO_\cX (Y_j ),\beta_{\mathrm{triv}})$, and therefore
  $E_\fm \simeq E$.

  (\ref{thm:main2}),(\ref{thm:main3}) We now have an exact sequence
  \[
    0 \to \cT_\fm \to P_\fm' /E_\fm \to P' /E \to 0
  \]
  of smooth separated $S$-algebraic spaces in groups, which are
  therefore separated $S$-group schemes \cite[(3.3.1)]{Ray70}, whose
  generic fibre is the sequence $0 \to T_\fm \to J_\fm \to J \to
  0$. As $\cT_\fm$ and $P'/E$ are the N\'eron models of $T$ and $J$,
  the result follows from Proposition \ref{prop:converse-to-neron-exact}.

  (\ref{thm:main4}) From \S\ref{sec:stuff} above,
  $\cT_{\fm,s} /\cT_{\fm,s}^0 \simeq \coker(e \colon \ZZ \to \ZZ^I
  )$. We then have a commutative diagram of \etale\ sheaves on $(Sm/S)$
  \[
    \begin{tikzcd}
      0 \arrow[r] & \cT_\fm \arrow[r] \arrow[d]
      & P_\fm \arrow[r] \arrow[d] & P \arrow[r] \arrow[d] & 0
      \\
      & s_* (\ZZ^I /e\ZZ) \arrow[r] & P_\fm /P_\fm^0 \arrow[r] & P/P^0 \arrow[r] & 0
    \end{tikzcd}
  \]
  whose rows are exact (since $\pi_0$ is right exact). For $S' /S$
  smooth, and $(\cL, \beta = (\beta_i ))$ representing an element of
  $\cF^* (S' )$, $\beta_i (1)$ is a rational section of
  $(g_i \times \id_{S'} )^* \cL$ so has a well-defined order along the
  special fibre $\ord_\cL \beta_i (1) \in \Gamma(S' , s_* \ZZ)$.  If
  $(\cL', \beta')$ is equivalent to $(\cL, \beta)$ then
  $(\ord_\cL \beta_i' (1) - \ord_{\cL'} \beta_i (1))_i \in \Gamma(S' , s_*
  (e\ZZ))$, which gives a splitting of the bottom row in the diagram
  (which is therefore also exact on the left).
\end{proof}

\begin{exam}
\label{ex:conic-picard}
Let’s work out the simplest nontrivial example: assume that
$\mathrm{char}(\F ) \ne 2$, and let $\cX$ be the closed subscheme of
$\mathbb{P}^2_R$ given by the equation $T_1 T_2 = \varpi T_0^2$. Then
$X = \cX_\F$ is a smooth conic, split over $\F$, and $\cX_s$ is the
line pair $T_1 T_2 = 0$. Hypotheses  (H1--3) of Section
\ref{sec:raynaud}  are all
satisfied. Let $x_0$, $x_1 \in X(\F ) = \cX(S)$ be distinct
points. Let $\cX_s = Y \cup Y' $, where the components are labelled in
such a way that $x_0$ meets $Y'$. We consider the generalized Jacobian
$J_\fm$ with $\fm = (x_0 ) + (x_1 )$.  The relative Picard space
$P = \Pic_{\cX/S}$ is a scheme, and is the union of its sections over
$S$. We have $P (\F ) = \ZZ$, generated by the class of
$\cO_\cX (x_0 )$, and $P (R) = P_s (k) = \ZZ^2$, generated by the
classes of $\cO_\cX (Y ) \simeq \cO_\cX (-Y' )$ and $\cO_\cX (x_0
)$. The restriction map $P (S) \to P (\F )$ is the second projection
$\ZZ^2 \to \ZZ$, and equals the degree map. Therefore $P' = E$ is the
``skyscraper scheme" $s_* \ZZ$, obtained by gluing copies of $S$
indexed by $\ZZ$ along their generic points, and $P' (S)$ is generated
by the class of $\cO_\cX (Y )$.

There is an isomorphism $\Gm \isom J_\fm = P_\Sigma' \otimes \F $,
which on $\F$-points takes $a \in \F^\mul$ to the equivalence class of
the pair $(\cO_\cX , \alpha = (\alpha_0 , \alpha_1 ))$, where
$\alpha_i \colon \F \to x_i^* \cO_\cX = \F$ is the identity for
$i = 0$ and multiplication by $a$ for $i = 1$. As $x_0$ doesn’t meet
$Y$ we also have $x_0^*\cO_\cX (Y ) = \cO_S $. We now have two cases:
\begin{itemize}
\item If $x_1$ meets $Y'$, then $x_1^*\cO_\cX (Y ) = \cO_S$ as well.
  So there is a canonical rigidification $(\alpha_i )$ of
  $\cO_\cX (Y )$ along $\Sigma$, for which each $\alpha_i$ is the
  identity map on $\cO_S$, and therefore
  $P_\Sigma \simeq \Gm \times P$ splits (and is not
  separated). Likewise, $P'_\Sigma\simeq \Gm \times s_*\ZZ$.
  The pushout $P_\fm'$ is simply the product $\cGm \times s_*\ZZ$.

\item If $x_1$ meets $Y$, then
  $x_1^*\cO_\cX (Y ) = \cO_S (s) = \varpi^{-1}\cO_S $. So there is a
  bijection $\ZZ \times R^\mul \isom P_\Sigma' (S)$ which takes
  $(n, a)$ to the line bundle $\cO_\cX (nY )$ with rigidification
  $\alpha_0 = \id$, $\alpha_1 (1) = \varpi^{-n}$. Its composition with
  restriction to the generic fibre is the bijection
  $\ZZ \times R^\mul \to P_\Sigma (\F ) = \F^\mul$ given by
  $(n, a) \mapsto \varpi^{-1} a$. So $P_\Sigma'$ is separated, and is
  isomorphic to the N\'eron model $\cGm$.
  The pushout $P_\fm'$ is then
  the tautological splitting of the extension
  $\Gm \to \cGm \to s_* \ZZ$ after pushing out through $\Gm \to \cGm$,
  so is isomorphic to $\cGm \times s_*\ZZ$ in this case as well.
\end{itemize}
\end{exam}

We return to the general case. From Section \ref{sec:raynaud}, $P_s^0 \cap E_s$ is finite
constant and cyclic 
of order $d$, generated by the class of the line bundle $\cL'$. Therefore
$P_{\fm,s}^0\cap E_{\fm,s}$
is finite constant and cyclic of order dividing $d$. Applying the
results of Section \ref{sec:singular2}, 
we obtain:
\begin{cor}
  \label{cor:charGenJac}
  Assume that $d = 1$. Then:
  \begin{enumerate}[\upshape (a)]
  \item
    \label{cor:charGenJac1}
    $\cJ_{\fm,s}^0  = \Pic^0_{(\cX_s ,\Sigma_s )/k}$.
  \item
    \label{cor:charGenJac2}
    If $k$ is perfect, there is a canonical isomorphism of
    $\Gal(\kbar/k)$-modules\mpa
    \[
      \Hom(\cJ_{\fm,s}^{0,\lin}\otimes_k \kbar, \Gm ) = H_1 (\Gtilde_{\cX_{\bar s}
        ,\Sigma_{\bar s}} , \ZZ)
    \]
    where the graph $\Gtilde_{\cX_{\bar s} ,\Sigma_{\bar s}}$ is as in
    Section \textup{\ref{sec:singular2}}.
\end{enumerate}
\end{cor}

Finally we compute the component group $\comp{J_\fm}$.

\begin{thm}
  \label{thm:phi_m}
  Suppose that $R$ is strictly Henselian. Then $\comp{J_\fm}$ is
  canonically isomorphic to the homology of the complex
  \begin{equation}
    \label{eq:raynaud2}
    \ZZ[C] \xrightarrow{(a,h)} \ZZ^C \oplus \ZZ^I /e\ZZ
    \xrightarrow{b\oplus0} \ZZ 
  \end{equation}
  where $a$ and $b$ are as in \eqref{eq:raynaud-defs}, and $h\colon \ZZ[C]
  \to \ZZ^I/e\ZZ$ is induced by the map
  \begin{align*} 
    C \times I &\to\ZZ \\
    (j, i) &\mapsto h_{ij} := \ord_{\cO_\cX(Y_j)} \beta_{\mathrm{triv},i} (1).
  \end{align*}
\end{thm}
\noindent(Equivalently, $h_{ij}$ is the degree of the divisor $g_i^* Y_j$ on
$\Spec R_i$.)
\begin{proof} By Theorem \ref{thm:main}, $\comp{J_\fm}$ is the group
  of connected components of the quotient $P_{\fm,s}'/E_{\fm,s}$,
  hence is the homology of the complex
  $E_\fm (k) \to \pi_0 (P_{\fm,s}) \xrightarrow{\deg}\ZZ$.  By 
Theorem \ref{thm:main}\eqref{thm:main4}, we may rewrite this complex as
  \eqref{eq:raynaud2}. What remains is to identify the map $h$. By the
  proof of \ref{thm:main}(\ref{thm:main1}), $E_\fm (k)$ is generated by the
  equivalence class of pairs
  $(\cO_\cX (Y_j ), \beta_{\mathrm{triv}} )$, and the proof of
  Theorem \ref{thm:main}(\ref{thm:main4}) then gives the desired 
formula for $h$.
\end{proof}

For general $S$ we have
$\Sigma \times_S \Ssh = \coprod_{\itilde \in \Itilde} S_{\itilde}$,
where $S_{\itilde}$ is the spectrum of a DVR finite over $\Rsh$. Let
$\Ctilde$ be the set of irreducible components of
$\cX\otimes\ks$. Then $\Gal(\F^{\mathrm{sep}} /\F )$ acts on $\Itilde$
and $\Ctilde$ through its quotient $\Gal(\ks/k)$, and the above gives
a $\Gal(\ks /k)$-equivariant isomorphism between $\comp{J_\fm}$ and
the homology of the complex
\begin{equation}
  \label{eq:raynaud2-tilde}
  \ZZ[\Ctilde] \xrightarrow{(a,h)} \ZZ^{\Ctilde} \oplus \ZZ^{\Itilde} /e\ZZ
  \xrightarrow{b\oplus0} \ZZ 
\end{equation}
attached to $\cX \times_S \Ssh$.

In the semistable case we can describe both the character and component groups
in terms of the reduced extended graph.

\begin{cor}
  \label{cor:semistable}
  Suppose that $\cX\otimes \Rsh$ is semistable and $I=I^\reg$. Then
\begin{enumerate}[\upshape (a)]
    \label{cor:semistable1}
\item If $k$ is perfect, there is a canonical isomorphism of
$\Gal(\kbar/k)$-modules
\[
\Hom(\cJ_{\fm,s}^{0,\lin}\otimes_k \kbar, \Gm ) = H_1 (\Gamma_{\cX_{\bar s}
  ,\Sigma_{\bar s}} , \ZZ)
\]
where the reduced extended graph $\Gamma_{\cX_{\bar s} ,\Sigma_{\bar s}}$ is as in
Section \textup{\ref{sec:singular2}}.
\item  \label{cor:semistable2} Assume that $R$ is strictly Henselian. There is a canonical isomorphism
\[
\comp{J_\fm} = \coker \left((\laplace, \theta^*) \colon \ZZ[C] \to
  \ZZ[C]\zero \oplus \ZZ^I
\right)
\] 
where $\laplace$ is the Laplacian of the reduced graph
$\Gamma_{\cX_s}$, and $\theta \colon I \to C$ is the map from
Section \textup{\ref{sec:singular2}}. 
\end{enumerate}
\end{cor}
Note that $\theta$  depends only on the labelled graph
$(\Gamma_{\cX_s,\Sigma_s}, \vzero)$. The hypothesis $I=I^\reg$ is
satisfied if for example $\{x_i\}\subset X(\Fsh)$.

\begin{proof}
  (a) follows immediately from Corollary \ref{cor:charGenJac}(\ref{cor:charGenJac2}) and the
  fact that the geometric realisations of $\Gamma_{\cX_s,\Sigma_s}$
  and $\Gtilde_{\cX_s,\Sigma_s}$ are homeomorphic. For (b), it is
  enough to observe that $(\laplace,\theta^*)$ maps $1\in \ZZ[C]$ to
  $(0,1)\in \ZZ[C]_0\oplus \ZZ^I$, and so the result follows from
  Theorem \ref{thm:phi_m}.
\end{proof}
\subsection{Description via N\'eron models of 1-motives \cite{Suz19}}
\label{sec:suzuki}

An alternative approach to the determination of the component group
$\comp{J_\fm}$ is via duality and the theory of N\'eron models of
$1$-motives developed in \cite{Suz19}. We recall some of the notions and
results of that paper. Recall that a $1$-motive over $\F$ is a two-term complex of
group schemes over $\F$
\[
  M = [\, L \xrightarrow{f} G \, ]
\]
where $L$ is \etale, free and finitely generated
(i.e.~$L\otimes_\F\F^{\mathrm{sep}}\simeq \ZZ^r$), and $G/\F$ is a
semiabelian variety. Let $T \subset G$ be its toric part, and $A=G/T$
the abelian variety quotient.  We assume here that $L$ and $T$ split
over an extension of $\F$ in which $R$ is unramified. Then $L$ extends to a local
system $\Lambda$ on $S$. Let $\cG$ be the N\'eron model of $G$. By the
N\'eron property, $f$ extends to a morphism $f_S\colon \Lambda\to \cG$
of $S$-group schemes, and by definition, the N\'eron model of $M$ is
the complex of $S$-group schemes
\[
  \mathcal{M} = [\, \Lambda \xrightarrow{f_S} \cG \,].
\]
Its component complex is the complex of $\Gal(\ks/k)$-modules
\[
\comp{M} = [\, \Lambda_{\bar s} \to \comp{G}\, ]
\]
in degrees $-1$ and $0$. (In \cite{Suz19} this complex is denoted
$\mathcal{P}(\mathcal{M})$.)

Let $M'$ be the $1$-motive dual to $M$. So
\[
  M' = [\, L' \xrightarrow{f'} G' \, ]
\]
where $L'=\Hom(T,\Gm)$ is the character group of $T$, and $G'$ is an
extension $T' \to G' \to A'$, where $A'$ is the dual abelian variety
of $A$, and $T'$ is the torus with character group $L$. Then
\cite[Theorem B]{Suz19} shows that if either
\begin{itemize}
\item[(i)] $A$ has semistable reduction, or
\item[(ii)] $k$ is perfect
\end{itemize}
there is a canonical isomorphism,
in the derived category of $\Gal(\ks/k)$-modules, between $\comp{M'}$ and
$\mathrm{RHom}(\comp{M},\ZZ)[1]$.

Now let $\cX/S$ be as in Section \ref{sec:hard-bit}.  We will assume
that $R$ is strictly Henselian. Suppose that all $n_j$ are zero (which
holds, for example, if $k$ is perfect), that $\delta=1$, and that the
points $(x_i)_{i\in I}$ are all $\F$-rational.

Since $J$ is autodual, the dual $1$-motive to $J_\fm$ is the $1$-motive
\[
  M = [\, \ZZ[I]\zero \to J\, ],\qquad i \mapsto \cO_X(x_i)
\]
whose component complex $\comp{M}$ is the complex
$[\, \ZZ[I]\zero \to \comp{J} \,]$ of abelian groups, concentrated in
degrees $-1$ and $0$. Using the description \eqref{eq:raynaud} of
$\comp{J}$, this is isomorphic to the complex
$[\, \ZZ[I]\zero \to \ZZ^{C,0}/a(\ZZ^C) \,]$.

As $\delta=1$, by \cite[(8.1.2)]{Ray70} the complex \eqref{eq:raynaud}
has only one nonzero homology group, namely $\ker(b)/\im(a)=\comp{J}$,
and the map $a$ is given by the intersection pairing on the components
of the special fibre. Therefore $\comp{M}$ is quasi-isomorphic to the
complex
\begin{equation}
  \label{eq:dual-complex}
  \begin{tikzcd}
    \ZZ \arrow[r, "{(i, 0)}"]
    & \ZZ[C]\oplus \ZZ[I]\zero \arrow[r, "{a\oplus {}^th}"] & \ZZ^C \arrow[r, "{b}"] & \ZZ.
\end{tikzcd}
\end{equation}
Here ${}^th\colon \ZZ[I]\zero \to \ZZ^C$ is the transpose of $h$, and the
term $\ZZ^C$ is in degree $0$. The dual of \eqref{eq:dual-complex} is
\[
  \begin{tikzcd}
    \ZZ \arrow[r, "{{}^tb}"] & \ZZ[C] \arrow[r, "{({}^ta, h)}"] & \ZZ^C \oplus \ZZ^I/\ZZ
    \arrow[r, "{{}^ti\oplus 0}"] & \ZZ
  \end{tikzcd}
\]
The assumption $n_j=0$ ensures that $a$ is symmetric, and that $i$ and
$b$ are transposes of one another, by
\eqref{eq:raynaud-defs}. Assuming that one of (i), (ii) above holds,
we then recover the description of $\comp{J_\fm}$ as the homology of
\eqref{eq:raynaud2}.

\subsection{Functoriality II}
\label{sec:functoriality2}

We will need to understand the action of correspondences on
generalized Jacobians and their N\'eron models.

Suppose that we have two smooth geometrically connected curves $X$,
$X'$ over $\F$, with regular models $\cX$, $\cX'$ satisfying the
hypotheses of \S\ref{sec:raynaud}. Let
$\fm=\sum_{i\in I}(x_i)$,
$\fm'=\sum_{j\in I'}(\xp_j)$, be nonzero moduli on $X$, $X'$. As in
\S\ref{sec:hard-bit} we assume that the points $x_i$, $\xp_j$ are
distinct, and that their residue fields
\[
F_i=\kappa(x_i),\quad F'_j = \kappa(\xp_j)
\]
are separable over $F$. Let $R_i$, $R'_j$ denote the integral closures
of $R$ in $F_i$, $F'_j$, and $\kModulus=\coprod_{i\in I}\Spec R_i$,
$\kModulus'=\coprod_{j\in I'}\Spec R_j'$.  Let $J_\fm$, $J'_{\fm'}$ be
the associated generalized Jacobians.

Let $f\colon X' \to X$ be a finite morphism 
such that $f^{-1}(\kModulus_F)=\kModulus'_F$ as sets. We write
$f\colon I' \to I$ also for the induced surjective map on index sets. For
$j\in I'$, denote by $r_j$ the ramification degree of $f$ at $\xp_j$.

The discussion in \S\ref{sec:functoriality1} applies, and since $f^*$, $f_*$ preserve line
bundles of degree zero, we obtain morphisms
\[
  f^* \colon J_\fm \to J'_{\fm'}, \quad f_* \colon J'_{\fm'} \to J_\fm\,.
\]
By the universal N\'eron property, they extend uniquely
to morphisms $f^*$, $f_*$ of the N\'eron models $\cJ_\fm$,
$\cJ'_{\fm'}$. Let the induced homomorphisms of character groups be
$\XX(f^*)$, $\XX(f_*)$ and of
component groups $\Phi(f^*)$, $\Phi(f_*)$. In the next section we will
need to know explicitly the restriction of these maps to the tori
$\cT_\fm$, $\cT'_{\fm'}$. For (a) and (b) below, recall that we have a
canonical isomorphism
\[
  \Hom(T_{\fm}\otimes \Fsep,\Gm) \isom \ZZ[\kModulus(\Fsep)]^{\deg=0}
\]
and similarly for $T'_{\fm'}$.

\begin{prop}  \mbox{ }
\label{prop:tori-funct}
\begin{enumerate}[\upshape (a)]
\item
\label{prop:tori-funct1}
  The map
  \[
    f^* \colon T_\fm = \Bigl(\prod_{i\in I} R_{\F_i/\F}\Gm\Bigr)/\Gm
    \to T'_{\fm'} = \Bigl(\prod_{j\in I'} R_{\Fp_j/\F}\Gm\Bigr)/\Gm 
  \]
  is induced by the inclusions $f^* \colon F_i \hookrightarrow F'_j$,
  $i=f(j)$. Its transpose is the homomorphism
  \[
    f_* \colon  \ZZ[\kModulus'(\Fsep)]^{\deg=0} \to\ZZ[\kModulus(\Fsep)]^{\deg=0}
  \]
  given by pushforward of divisors of degree zero.
\noop{
  \[
    \begin{tikzcd}
      \Hom(T'_{\fm'}\otimes \Fsep,\Gm) \arrow[r] \arrow[d, equal]&
      \Hom(T_{\fm}\otimes \Fsep,\Gm) \arrow[d, equal]  \\
      \ZZ[\kModulus'(\Fsep)]^{\deg=0} \arrow[r, "f"] &
      \ZZ[\kModulus(\Fsep)]^{\deg=0}\, .
    \end{tikzcd}
  \]
  }
\item
\label{prop:tori-funct2}
  The map $f_*\colon T'_{\fm'} \to T_\fm$ is given by the morphisms of
  tori, for $i=f(j)$,
  \begin{align*}
    \R_{\Fp_j/\F_i}\Gm & \to \R_{\F_i/\F}\Gm \\
    t &\mapsto (N_{\Fp_j/\F_i}t)^{r_j}\,.
  \end{align*}
  Its transpose is the homomorphism
  \[
    f^* \colon \ZZ[\kModulus(\Fsep)]^{\deg=0} \to
    \ZZ[\kModulus'(\Fsep)]^{\deg=0}
  \]
  given by pullback of divisors.
\item
\label{prop:tori-funct3}
Assume that $R=\Rsh$. Then the induced maps between component groups
$\comp{T_\fm}$, $\comp{T'_{\fm'}}$ are
  \[
    \begin{tikzcd}
      \llap{$\Phi(f^*) \colon\ $} \comp{T_\fm}
      = \ZZ^I/e\ZZ  \arrow[r]
      & \comp{T'_{\fm'}} \simeq  \ZZ^{I'}/e'\ZZ \\[-4ex]
      (n_i)_{i\in I}  \arrow[r, mapsto] &\left(
        (e'_j/e_{f(j)})n_{f(j)}\right)_{j\in I'} \\
      \llap{$\Phi(f_*) \colon \qquad$}\ZZ^{I'}/e'\ZZ 
      \arrow[r] &  \ZZ^{I}/e\ZZ \\[-4ex]
      (n_j)_{j\in I'} \arrow[r, mapsto]&
      \displaystyle{\Bigl(\sum_{j\in f^{-1}(\{i\})} r_j n_j\Bigr)_{i\in I}}
    \end{tikzcd}
  \]
\item
\label{prop:tori-funct4}
  Assume that $R=\Rsh$, and that $k$ is algebraically closed \textup(\!so that
  $\kModulus(k)\simeq I$, $\kModulus'(k)\simeq I'$\textup). Then the induced
  maps on character groups of N\'eron models are
  \[
    \begin{tikzcd}
      \llap{$\XX(f^*) \colon\ $}
      \ch{T'_{\fm'}} \arrow[r] \arrow[d, equal]
      & \ch{T_\fm} \arrow[d, equal]\\[-2ex]
       \ZZ[I']^{\deg=0} \arrow[r, "f"] & \ZZ[I]^{\deg 0} \\
      \llap{$\XX(f_*) \colon\ $}  \ZZ[I]^{\deg=0} \arrow[r, "f"] & \ZZ[I']^{\deg 0} \\[-4ex]
      r(i) \arrow[r, mapsto] & \displaystyle{\sum_{j\in f^{-1}(\{i\})}} r_j[\Fp_j:\F_i](j)      
    \end{tikzcd}
  \]
\end{enumerate}
\end{prop}

\begin{proof}
 For (\ref{prop:tori-funct1}) and (\ref{prop:tori-funct2}), it
 suffices to compute the map on character groups. The formulae are
 then special cases of Propositions \ref{prop:upper-star} and
 \ref{prop:lower-star} with $A=B=\Sigma^{\sing}=\emptyset$, $C=\{*\}$.

 Combining these with Proposition \ref{prop:tori-neron} then gives the
 remaining parts.

\end{proof}

\begin{rem}
  From (a) and (b) we see that if $f$, $f'\colon X'\to X$ are finite
  morphisms and $\fm$ is a reduced modulus on $X$ which is stable
  under the correspondence $A=f_*f^{'*}$ (in the sense of Example
  \ref{ex:functoriality1}), then the induced endomorphism ${}^tA$ of
  the character group $\Hom(T_\fm\otimes\Fsep,\Gm)$ equals the map
  $D \mapsto f'_*f^*D$ on divisors of degree zero.
\end{rem}

\section{Generalized Jacobians of modular curves}

\subsection{Generalities on modular curves}
\label{sec:mod-gen}

For an integer $N\ge1$, let $X_0(N)_\QQ$ denote the usual complete
modular curve over $\QQ$. Its non-cuspidal points parametrize pairs
$(E,C)$, where $E$ is an elliptic curve (over some $\QQ$-scheme) and
$C\subset E$ is a subgroup scheme which is cyclic
of order $N$. We write $X_0(N)_\ZZ$ for the integral model constructed
by Katz and Mazur \cite[Ch. 8]{KM85}, which they denote $\overline
M([\Gamma_0(N)])$. Its non-cuspidal points parametrize pairs 
$(E,C)$, where $C\subset E$ is a subgroup scheme of rank $N$ which is
cyclic in the sense of \emph{loc.~cit.}, \S6.1 --- see also \cite[\S1.1]{Edi90}.  

For every prime $\ell$ we have a Hecke correspondence $T_\ell=v_*u^*$,
where the finite morphisms $u=u_\ell$, $v=v_\ell$ are given by:
\[
  \begin{tikzcd}[column sep=small]
    & X_0(N\ell) \arrow[ld, "u"'] \arrow[rd, "v"]
    &&& (E,C) \arrow[ld, mapsto, "u"'] \arrow[rd, mapsto, "v"] \\
    X_0(N) && X_0(N) & (E,\ell C)&& \!\!\!\!(E/NC,C/NC)\!\!\!\!
  \end{tikzcd}
\]
For $\ell\notdivides N$ (resp.~$\ell\divides N$), the morphisms $u$, $v$ are
of degree $\ell+1$ (resp.~$\ell$). For $p\divides N$ we also have the
Atkin-Lehner involution $W_p\colon X_0(N) \to X_0(N)$.
If $v_p(N)=r\ge1$ then 
\[
  W_p\colon (E,C) \mapsto
  (E/(C\cap E[p^r]), (C+E[p^r])/ (C\cap E[p^r])).
\]
When $\ell\notdivides N$, $u=v\circ W_\ell$. and the correspondence
$T_\ell$ is symmetric. When $\ell\divides N$, $T_\ell$ is no longer
symmetric, and what we call $T_\ell$ is often elsewhere defined to be
the transpose of $T_\ell$ (and also often written $U_\ell$). We have
chosen our normalisations so that the endomorphism $T_\ell=v_*u^*$ of the Jacobian $J_0(N)_\QQ$
agrees with the Hecke operator in \cite[p. 445]{Rib90} defined by
``Picard functoriality''.

Write $X_0(N)^\infty_\QQ \subset X_0(N)_\QQ$ for the cuspidal
subscheme. It is classical that $X_0(N)^\infty_\QQ$ is the disjoint
union, over positive divisors $d\divides N$, of schemes
$z_d\simeq\Spec \QQ(\mmu_{(d,N/d)})$ (where here $(d,N/d)$
  denotes greatest common divisor).  We recall (e.g.~from
\cite[IV.4.11--13]{DeRa72}) that the cusps of $X_0(N)_\QQ$ can be
conveniently described using generalized elliptic curves. Suppose that
$d\divides N$, and let $\Ner_d$ denote the standard N\'eron polygon over
$\QQ$ with $d$ sides \cite[II.1.1]{DeRa72}, whose smooth locus
$\Ner_d^{\reg}$ equals $\Gm\times\ZZ/d$. For a primitive $N$-th root
of unity $\zeta_N\in\overline\QQ$, let $\C{\zeta_N}$ denote the cyclic
subgroup scheme
\[
  \C{\zeta_N}=\langle(\zeta_N,1)\rangle \subset \Ner_d^\reg.
\]

Then the pair
$(\Ner_d,\C{\zeta_N})$ determines a $\QQbar$-point of $X_0(N)^\infty$. If
$\sigma\in\Gal(\QQbar/\QQ(\mmu_d))$, then $\C{\sigma\zeta_N}=
\C{\zeta_N}$, and if $\sigma\in\Gal(\QQbar/\QQ(\mmu_{(d,N/d)}))$ then
the pairs $(\Ner_d,\C{\sigma\zeta_N})$ and $(\Ner_d,\C{\zeta_N})$ are isomorphic (and
the isomorphism is unique if it is required to be the identity on the
identity component of $\Ner_d$). Therefore the isomorphism class of
$(\Ner_d,\C{\zeta_N})$ over $\QQbar$ is determined by the pair
$(d,\zeta_N^{N/(d,N/d)})$ and gives rise to
a closed point $z_d=z_{N,d}\simeq\Spec \QQ(\mmu_{(d,N/d)})$ of $X_0(N)^\infty_\QQ$.

In particular, the (rational) cusps $\infty=z_{N,1}$ and $0=z_{N,N}$ correspond
to the pairs $(\Ner_1,\mmu_N)$ and $(\Ner_N,\{1\}\times \ZZ/N)$,
respectively. We also know that the scheme-theoretic closure of
$X_0(N)^\infty_\QQ$ in $X_0(N)_\ZZ$ is the disjoint union of copies of
$\Spec \ZZ(\mmu_{(d,N/d)})$ (this follows from
\cite[Thm. 1.2.2.1]{Edi90}).

Now let $\fm$ be a reduced modulus on $X_0(N)_\QQ$, whose support is
contained in $X_0(N)_\QQ^\infty$. Let $\ell$ be any prime such that the
support of $\fm$ is stable under $T_\ell$, in the sense of Example
\ref{ex:functoriality1}.  Then $T_\ell$ determines an endomorphism
$T_\ell=v_*u^*$ of $J_\fm$. Let $p\divides N$, and let $\cJ_\fm$ be the N\'eron model of
$J_\fm$ over $\ZZ_{(p)}$. By the universal N\'eron
property, $T_\ell$ extends to an endomorphism of $\cJ_\fm$, and
therefore induces endomorphisms
\begin{align*}
  T_\ell \colon \comp{J_\fm} &\to \comp{J_\fm} \\
  {}^tT_\ell \colon  \ch{J_\fm} &\to \ch{J_\fm}.
\end{align*}
In order to compute these endomorphisms combinatorially, we need to
compute the action of $T_\ell$ on the torus $\cT_\fm$, using the
formulae of Proposition \ref{prop:tori-funct}.  In other words, we
need to compute the restrictions of $u=u_\ell$, $v=v_\ell$ to the
cusps, along with the ramification degrees.

Let $\zeta_{N\ell}\in\QQbar$ be a primitive $N\ell$-th root of unity,
and for $L\divides N\ell$, $\zeta_L=\zeta_{N\ell}^{N\ell/L}$. Let
$z=(\Ner_d,\C{\zeta_{N\ell}}) \in X_0(N\ell)^\infty(\QQbar)$ be a
cusp. Then $u(z)\in X_0(N)^\infty(\QQbar)$ is obtained as follows:
replace $\C{\zeta_{N\ell}}$ by
$\ell\C{\zeta_{N\ell}}=\langle(\zeta_N,\ell)\rangle\subset
\Gm\times\ZZ/d$, and then contract any components of $\Ner_d$ which do
not meet it \cite[IV.1.2]{DeRa72}. Similarly, we obtain $v(z)$ as the
quotient of $(\Ner_d,\C{\zeta_{N\ell}})$ by the rank-$\ell$ group scheme
$N\C{\zeta_{N\ell}}=\langle (\zeta_\ell,N) \rangle \subset\Gm\times\ZZ/d$.

Explicitly, suppose that $N=M\ell^k$, $(\ell,M)=1$, and that
$d\divides N\ell$. Let $a$, $b\in\ZZ$ with $a\ell+bM=1$ and
$a\equiv 1$ (mod~$\ell^k$). Then if $\ell\notdivides d$,
\[
(\Ner_d,\langle(\zeta_N,\ell)\rangle)=(\Ner_d,\C{\zeta_N^a})
\]
but if $\ell\divides d$, the subgroup $\langle(\zeta_N,\ell)\rangle)$
does not meet the components $\Gm\times\{i\}$ with
$(i,\ell)=1$. The map $\Ner_d \to \Ner_{d/\ell}$ contracting them 
takes $\langle(\zeta_N,\ell)\rangle$ to $\C{\zeta_N}$, and so
\[
u\colon (\Ner_d,\C{\zeta_{N\ell}}) \mapsto
\begin{cases}
  (\Ner_d,\C{\zeta_N^a}) &\text{if $\ell\notdivides d$} \\
  (\Ner_{d/\ell},\C[d/\ell]{\zeta_N}) &\text{otherwise.}
\end{cases}
\]
On closed points of $X_0(N)_\QQ^\infty$ we then have
\begin{equation}
  \label{eq:ucusps}
u^{-1}(z_{N,d}) = \begin{cases}
\{ z_{N\ell, d\ell} \} & \text{if $\ell | d$} \\
\{ z_{N\ell, d},\ z_{N\ell, d\ell} \} & \text{otherwise.}
\end{cases}
\end{equation}
If $\ell\not\!|N$ then $u$ has degree $\ell+1$ and
\[
\deg z_{N,d} = \deg z_{N\ell,d} = \deg z_{N\ell,d\ell} = \phi((d, N/d)).
\]
It is well known (and follows, for example, from the Eichler-Shimura congruence relation) that $u$ is \'etale at $z_{N\ell,d}$, and so has ramification degree $\ell$ at $z_{N\ell,d\ell}$.

Suppose now that $k\ge1$ and $d_0 | M$, $d=d_0\ell^s$. Then $u$ has degree $\ell$ and
\begin{align*}
\deg z_{N\ell,d} &=
\phi(\ell^{\min(s,k+1-s)}) \phi((d_0,M/d_0)) \\
\deg z_{N\ell,d/\ell} &=
\phi(\ell^{\min(s-1,k+1-s)}) \phi((d_0,M/d_0))\quad\text{if $s\ge1$}
\end{align*}
so by \eqref{eq:ucusps}, the ramification degree of $u$ at $z_{N\ell,d}$ equals
\begin{align*}
1 &\quad \text{if $1< s \le (k+1)/2$, and} \\
\ell &\quad \text{if $(k+1)/2 < s \le k+1$.}
\end{align*}
Moreover, since
\begin{gather*}
\deg z_{N\ell, d_0\ell}  = (\ell-1)\phi((d_0,M/d_0)) \\
\deg z_{N\ell, d_0}  =\deg z_{N, d_0}  =\phi((d_0,M/d_0))
\end{gather*}
the ramification degree equals $1$ also for $s\in\{0,1\}$.

Similarly, if $d\divides N$, then the subgroup
$N\C{\zeta_{N\ell}}\subset \Ner_d^\reg=\Gm\times\ZZ/d$ equals
$\mmu_\ell\times\{0\}$, and therefore is the kernel
of the endomorphism $(t,i)\mapsto (t^\ell,i)$ of $\Ner_d$. If
$d\divides N\ell$ but $d\notdivides N$ then
$N\C{\zeta_{N\ell}}=\langle(\zeta_\ell,N)\rangle$ is the kernel of the
map $\Ner_d\to \Ner_{d/\ell}$, $(t,i)\mapsto
(t\zeta_{\ell^k+1}^{-bi},i\text{ mod }d/\ell)$, which maps
$(\zeta_{N\ell},1)$ to $(\zeta_N^a,1)$, and therefore
\[
v\colon (\Ner_d,\C{\zeta_{N\ell}}) \mapsto
\begin{cases}
  (\Ner_d,\C{\zeta_N}) &\text{if $d\divides N$} \\
  (\Ner_{d/\ell},\C[d/\ell]{\zeta_N^a}) &\text{otherwise}.
\end{cases}
\]
A similar computation as for $u$ shows that the ramification degree of
$v$ at $z_{N\ell,d}$ is $1$ if $v_\ell(d)\ge (k+1)/2$, and  $\ell$
otherwise.

In particular, if $\fm$ is any reduced modulus supported on
$X_0(N)^\infty_\QQ$, then for every $\ell\notdivides N$, $\fm$ is
stable under $T_\ell$ (in the sense of Example
\ref{ex:functoriality1}) and therefore we obtain an endomorphism
$T_\ell=v_*u^*$ of the generalized Jacobian $J_\fm=J_0(N)_\fm$. If
$\fm$ is the full cuspidal modulus (i.e.~the reduced modulus whose
support is $X_0(N)^\infty_\QQ$) then $\fm$ is stable under $T_\ell$
for every $\ell$. Using the formulae from Proposition
\ref{prop:tori-funct} together with the fact that $(\Ner_d,\C{\zeta_N})$ depends
only on $(d,\zeta_N^{N/(d,N/d)})$, we can compute the induced endomorphism ${}^tT_\ell$ of the character group
\[
  \Hom(T_\fm \otimes\QQbar,\Gm)=\ZZ[X_0(N)^\infty(\QQbar)]^{\deg=0}
\]
which is the restriction of $u_*v^*$ to divisors of degree zero.

\begin{prop}
  \label{prob:Tl-tori}
  \textup{(a)} If $(\ell,N)=1$, then 
  \[
    {}^tT_\ell(\Ner_d,\C{\zeta_N})=(\Ner_d,\C{\zeta_N^\ell})+\ell(\Ner_d,\C{\zeta_N^a}).
  \]
  \textup{(b)} If $N=M\ell^k$ with $(M,\ell)=1$ and $k>0$, and $v_\ell(d)=i$, then let
  $d=d_0\ell^i$, $e_0=(d_0,M/d_0)$,
  $\Gamma_i=\Gal(\QQ(\mmu_{e_0\ell^{k+1-i}})/\QQ(\mmu_{e_0\ell^{k-i}}))$. Then 
  \[
    {}^tT_\ell(\Ner_d,\C{\zeta_N})=
    \begin{cases}
      \ell(\Ner_d,\C{\zeta_N^a}) & i=0\\
      \ell(\Ner_{d/\ell},\C[d/\ell]{\zeta_N}) & 0<i < (k+1)/2\\
      \sum_{\sigma\in\Gamma_i}\sigma(\Ner_{d/\ell},\C[d/\ell]{\zeta_N})
      & (k+1)/2\le i < k \\
      \sum_{\sigma\in\Gamma_k}\sigma(\Ner_{d/\ell},\C[d/\ell]{\zeta_N})
      +(\Ner_d,\C{\zeta_N^q}) & i=k 
    \end{cases}
  \]
  where $a$, $b$ are as above, and $aq\equiv 1$ \textup{(mod $N$)}.
\end{prop}
(In (b) $\Gamma_i\simeq (\ZZ/\ell\ZZ)^\mul$ if $i=k$ and $\ZZ/\ell\ZZ$
otherwise, so consistent with $\deg{}^tT_\ell=\ell$.)

\begin{proof}    
  First note that if $v_\ell(d)=k+1$ then
  \[
    v\colon (\Ner_d,\C{\zeta_{N\ell}^q}) \mapsto
    (\Ner_{d/\ell},\C[d/\ell]{\zeta_N}). 
  \]
  \boxed{k=0} Then
  \[
    v^*(\Ner_d,\C{\zeta_N}) = \ell(\Ner_d,\C{\zeta_{N\ell}}) +
    (\Ner_{d\ell},\C[d\ell]{\zeta_{N\ell}^q})
  \]
  hence (since $q\equiv \ell$ mod $N$ when $k=0$)
  \[
    {}^tT_\ell(\Ner_d,\C{\zeta_N}) = \ell(\Ner_d,\C{\zeta_{N}^a})
    + (\Ner_{d},\C[d\ell]{\zeta_{N}^\ell}).
  \]
  \boxed{k>0} Then if $v_\ell(d)<(k+1)/2$,
  \[
    v^*(\Ner_d,\C{\zeta_N}) = \ell(\Ner_d,\C{\zeta_{N\ell}})
  \]
  and applying $u_*$ to this gives $\ell(\Ner_d,\C{\zeta_{N}^a})$.
  
  If $(k+1)/2\le v_\ell(d) <k$ then the inverse image of the cusp
  $(\Ner_d,\C{\zeta_N})$ is the union of $\ell$ cusps conjugate to
  $(\Ner_d,\C{\zeta_{N\ell}})$, namely
  \[
    v^*(\Ner_d,\C{\zeta_N}) = \sum_{\sigma\in\Gamma_i}
    \sigma(\Ner_d,\C{\zeta_{N\ell}}).
  \]
  Finally if $v_\ell(d)=k$ then
  \begin{align*}
    v^*(\Ner_d,\C{\zeta_N}) = &\sum_\sigma
    \sigma(\Ner_d,\C{\zeta_{N\ell}}) \qquad(\ell-1 \text{ terms}) \\
    &\quad + (\Ner_{d\ell},\C[d\ell]{\zeta_{N\ell}^q})
  \end{align*}
  Apply $u_*$ to this and we get the claimed formula. 
\end{proof}

\begin{exam}
  \label{ex:hecke-calcs}
  Set $D=(0)-(\infty)$. Here are particular cases we will need:
\begin{enumerate}[\upshape (a)]
\item \label{ex:hecke-calcs-Tl} $(\ell,N)=1$,
  $\fm=(\infty)+(0)=(\Ner_1,\mmu_N)+(\Ner_N,\ZZ/N)$. Then
  \[
    {}^tT_\ell\colon D\mapsto (\ell+1)D.
  \]
\item \label{ex:hecke-calcs-p} $N=p$ prime, $\fm=(\infty)+(0)$. Then
  ${}^tT_p\colon D\mapsto D$.
\item \label{ex:hecke-calcs-p2} $N=p^2$. There are $(p+1)$ elements of
  $X_0(p^2)^\infty(\QQbar)$:
  \[
    (\infty) = (\Ner_1,\C[1]{\zeta_{p^2}}),
    \ (0) = (\Ner_{p^2},\C[p^2]{\zeta_{p^2}}),
    \ (\zeta_p)= (\Ner_{p},\C[p]{\zeta_{p^2}})\ (1\ne\zeta_p\in\mmu_p).
    \]
  Then if
  $\ell\ne p$,
  \begin{align*}
{}^tT_\ell\colon\ D
    & \mapsto (\ell+1)D\\
    (\zeta_p)-(\infty)
    & \mapsto \ell(\zeta_p^{1/\ell})
      +(\zeta_p^\ell)
      -(\ell+1)(\infty)
   \intertext{and}
     {}^tT_p\colon\ D
    & \mapsto\sum_{1\ne \zeta_p\in\mmu_p}(\zeta_p)+(0)-p(\infty)\\
    (\zeta_p)-(\infty)
    & \mapsto 0.
  \end{align*}
\end{enumerate}
\end{exam}

\subsection{Character groups}

Assume that $N=pM$, with $p>3$ prime and $(p,M)=1$. Let $\SPSM$ be the
set of supersingular points of $X_0(M)(\FFpbar)$, which is the set of
isomorphism classes of pairs $(E,C)$, where $E/\FFpbar$ is a
supersingular elliptic curve and $C\subset E$ is a cyclic subgroup
scheme of order $M$. 

For $\ell\notdivides M$, we have the Hecke operator
\begin{align*}
  T_\ell \colon \ZZ[\SPSM] & \to \ZZ[\SPSM] \\
  (E,C) &\mapsto \sum_{D\subset E,\ \#D=\ell} (E/D,(C+D)/D).
\end{align*}

\begin{thm}
  \label{thm:characters-supersingular}
  Let $\fm=(\infty)+(0)$ and $J=J_0(N)$ with $N=pM$ as above. Then
  there is a canonical isomorphism
  \[
    \ch{J_\fm} \isom \ZZ[\SPSM]
  \]
  taking ${}^tT_\ell$  to $T_\ell$ for every
  $\ell\notdivides N$. Its restriction to $\ch{J} \hookrightarrow
  \ch{J_\fm}$ is an isomorphism $\ch{J} \isom \ZZ[\SPSM]\zero$.
\end{thm}

(The second isomorphism is of course well known: see
\cite[Prop. 3.1]{Rib90}.)

\begin{proof}
  We work over $S$, the strict henselisation of $\Spec \ZZ_{(p)}$, and
  use the notations from \S2, so that $k=\FFpbar$. Let $\cX'$ denote
  the Deligne-Rapoport model of $X_0(N)$ over $S$. Since $p$ exactly
  divides $N$, $\cX'$ is regular apart from possible $A_2$ or $A_3$
  singularities at supersingular points in the special fibre where
  $j=0$ or $1728$. Let $\cX \to \cX'$ be its minimal
  desingularisation. The special fibre $\cX'_s$ is the union of two
  copies of the modular curve $X_0(M)_{\FFpbar}$ meeting
  transversally at the supersingular points. The cusp $\infty$
  (resp.~$0$) meets the component of $\cX'_s$ parametrizing $(E,C)$
  where $C$ contains the kernel of Frobenius (resp.~Verschiebung). Let
  us refer to these as the  $\infty$-component $Z_\infty$ and
  $0$-component $Z_0$ of $\cX_s$.
  
  First we assume that $\cX'=\cX$ is regular (which holds, for
  example, if $M$ is divisible by some prime $q\equiv -1$~(mod~$12$)
  or by $36$ --- see the second table in \cite[4.1.1]{Edi91}). Since
  $\cX_s$ has an irreducible component of multiplicity one,  the
  hypotheses (H1--3) of Section \ref{sec:raynaud} are satisfied. Let
  $\kModulus \to \cX$ be the morphism induced by $\fm$, so that
  $\kModulus$ is the disjoint union of two sections of $\cX$ over $S$.
  Then\mpa
  \begin{equation}
    \label{eq:modular-raynaud}
    J_{\fm/\FFpbar}^0 \simeq \Pic^0_{(\cX_s,\kModulus_s)/k}
  \end{equation}
  and since $\Sigma_s\subset \cX^{\reg}_s$, by \eqref{eq:pizza1} we have 
  \begin{equation}
    \label{eq:modular-char1}
    \ch{J_\fm} = \ker\Bigl[
    \ZZ[\widetilde{\SPS}]\oplus \ZZ[\Sigma_s]
    \xrightarrow{\left[\begin{smallmatrix}\psi&\theta\\ 
\phi&0\end{smallmatrix}\right]}
    \ZZ[C] \oplus \ZZ[\SPS] \Bigr]
  \end{equation}
  where $C=\pi_0(\widetilde{\cX_s})$, $\SPS=\SPSM\simeq \cX_s^{\sing}$
  and $\widetilde{\SPS}$ is the inverse image of $\SPS$ in the
  normalisation $\widetilde{\cX_s}$ of $\cX_s$.
  
  The map $\theta\colon \ZZ[\Sigma_s] \to \ZZ[C]$ is a bijection
  since the cusps meets different components, and $\cX_s$ has only
  ordinary double points, so the vertical maps between the three $2$-term complexes
\[
    \begin{tikzcd}
      \ZZ[\widetilde{\SPS}]\oplus \ZZ[\kModulus_s] \arrow[r]  \arrow[d] 
      & \ZZ[C] \oplus \ZZ[\SPS] \arrow[d] \\
      \ZZ[\widetilde{\SPS}] \arrow[r] & \ZZ[\SPS] \\
      \ZZ[\SPS] \arrow[r] \arrow[u, "i"] & 0 \arrow[u]
    \end{tikzcd}
  \]
  are quasi-isomorphisms. Here $i$ is the map taking $x\in \SPS$ to
  $x^{(\infty)}-x^{(0)}$, with $x^{(\infty)}$, $x^{(0)}\in \cX_s$
  being the supersingular points above $x$ lying in the components
  containing $\infty$, $0$ respectively.  These quasi-isomorphisms
  then induce the isomorphism $\ch{J_\fm}\simeq \ZZ[\SPS]$. To get
  $\ch{J}$ we drop the factor $\ZZ[\kModulus_s]$ from
  \eqref{eq:modular-char1}, and then the kernel becomes $\ZZ[\SPS]\zero$.

  Still assuming that $\cX'$ is regular, let $\ell\ne p$ be
  prime. Then the Deligne-Rapoport model for $X_0(N\ell)$ over $S$ is
  also regular. Let us denote it $\cX^{(\ell)}$. Then maps $u,v$
  extend to finite morphisms $\cX^{(\ell)} \to \cX$ which are
  therefore also flat. Therefore the endomorphism $T_\ell$ of
  $\cJ_{\fm,s}^0$ is, under the isomorphism
  \eqref{eq:modular-raynaud}, identified with the endomorphism
  $T_\ell=u_*v^*$ of $\Pic^0_{(\cX_s,\kModulus_s)/k}$. Now the maps
  $u$, $v \colon \cX^{(\ell)}_s \to \cX_s$ map the $\infty$- and
  $0$-component of $\cX^{(\ell)}_s$ to the $\infty$- and
  $0$-component, respectively, of $\cX_s$, and on each of these, they
  are just the maps $u$, $v\colon X_0(M\ell) \to X_0(M)$. So $u_*v^*$
  induces the map $T_\ell$ on $\ZZ[\SPSM]$.

  In general, choose a multiple $N'=nN=pM'$ of $N$ with $(p,n)=1$ such
  that the Deligne-Rapoport model of $X_0(N')$ over $S$ is
  regular. Let $f\colon X_0(N') \to X_0(N)$ be the map
  $(E,C)\mapsto (E,nC)$, and $\fm'$ the reduced modulus
  $f^{-1}((\infty)+(0))^{\red}$ on $X_0(N')$. Then
  $f^*\colon J_0(N)_{\fm} \to J_0(N')_{\fm'}$ induces a surjection
  \[
    \begin{tikzcd}
      \llap{${}^tf^* \colon$} \ch{J_0(N')_{\fm'}} \arrow[r] & \ch{J_0(N)_\fm} \\
      \ZZ[\SPSMp] \arrow[u, "\wr"] & \ZZ[\SPSM] \arrow[u, "\wr"] 
    \end{tikzcd}
  \]
  which is equivariant with respect to ${}^tT_\ell$ for all
  $\ell\notdivides N'$. According to \ref{prop:upper-star} this is
  induced by the map $f\colon \SPSMp \to \SPSM$, hence commutes
  with the maps $T_\ell$ on $\ZZ[\SPSMp]$ and $\ZZ[\SPSM]$.
\end{proof}
  
\begin{rem}
  Restricting to the case when $p$ exactly divides $N$ is rather
  natural, since the toric part of the special fibre of the N\'eron
  model of $J_0(p^rM)$, $r>1$, is a product of copies of the toric
  part for $J_0(pM)$.
\end{rem}

In the case $N=p$ we may describe everything (including $T_p$) in
terms of the classical Brandt matrices, whose definition we now recall
\cite{Gros87}.  Let $\{ E_i \mid 1\le i\le h\}$ be representatives of the
isomorphism classes of supersingular elliptic curves over
$\FFpbar$ (so that $h$ is the class number of the definite quaternion
algebra $\mathrm{End}(E_i)\otimes\QQ$). Let $\Hom(E_i,E_j)_n$ be the set of isogenies from $E_i$ to
$E_j$ of degree $n$.  Define an equivalence relation $\sim$ on
$\Hom(E_i,E_j)_n$ by
\[
f\sim g \Longleftrightarrow \ker f=\ker g\Longleftrightarrow
f=\alpha g\text{ for an automorphism $\alpha$ of $E_j$}
\]
and set $\overline{\Hom}(E_i,E_j)_n=\Hom(E_i,E_j)/\sim$.
We then define the $h\times h$  Brandt matrix $B(n)$ for $n\geq 1$ by
\begin{equation}
\label{eq:brandt}
B(n)_{ij}=\#\overline{\Hom}(E_i, E_j)_n.
\end{equation}
The matrices $B(n)$ for $n\ge 1$ commute. They are constant row-sum
matrices, with the sum of the entries in any row of $B(n)$ equal to
\[
  \sigma'(n)=\smash[b]{\sum_{d\divides n,\ (p,d)=1} d}
\]
for $n\geq 1$.

\begin{thm}
  \label{thm:X0p-char}
  Let $N=p$ and $\fm=(\infty)+(0)$. The isomorphism
  $\ch{J_\fm}\isom \ZZ[\SPS_1]$ of Theorem
  \textup{\ref{thm:characters-supersingular}} takes  ${}^tT_\ell$  to the
  transpose ${}^tB(\ell)$ of the Brandt matrix, for every prime $\ell$
  \textup(including $\ell=p$\textup).
\end{thm}

\begin{proof}
  For $\ell\ne p$ this follows immediately from the definition of the
  Brandt matrix $B(\ell)$. For $\ell=p$, we first note that the
  endomorphism $T_p+W_p$ of $J_0(p)_\fm$ is zero. Indeed, on the
  quotient $J_0(p)$ it is zero, by \cite[Proposition 3.7]{Rib90}, and
  since $W_p$ interchanges the two cusps and ${}^tT_p$ fixes
  $(0)-(\infty)$, it is zero on the torus
  $T_\fm=\Gm\subset J_0(p)_\fm$. So as any morphism from $J_0(p)$ to
  $\Gm$ is constant, $T_p+W_p$ is zero on $J_0(p)_\fm$. Therefore it
  is enough to compute the action of $W_p$ in $\ch{J_0(p)_\fm}$. For
  this it is convenient to compute using the extended reduced graph
  $\Gamma_{\cX'_s,\Sigma}$ defined in \S\ref{sec:singular2}, with
  $\Sigma=X_0(p)^\infty_s=\{\infty,0\}$ (where we have fixed an
  orientation):
  \begin{center}
    \tikzset{middlearrow/.style={
        decoration={markings,
          mark= at position 0.5 with {\arrow{#1}} ,
        },
        postaction={decorate}
      }
    }
    \scalebox{1.2}
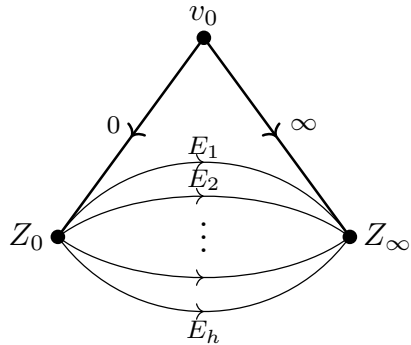
{
      \begin{tikzpicture}
        \draw [black, thick, middlearrow={>}] (1.6,2.2) -- (0,0);
        \draw [black, thick, middlearrow={>}] (1.6,2.2) -- (3.2,0);
        \draw[middlearrow={>}] (0,0) .. controls (.8, 1.1) and (2.4, 1.1) .. (3.2,0);
        \draw[middlearrow={>}] (0,0) .. controls (.8, .6) and (2.4,.6) .. (3.2,0);
        \draw[middlearrow={>}] (0,0) .. controls (.8, -.6) and (2.4, -.6) .. (3.2,0);
        \draw[middlearrow={>}] (0,0) .. controls (.8, -1.1) and (2.4, -1.1) .. (3.2,0);
        \node [left] at (0,0) {\footnotesize $Z_0$};
        \node[right] at (3.2,0) {\footnotesize $Z_\infty$};
        \node[above] at (1.6,2.2) {\footnotesize $v_0$};
        \fill[black] (1.6,2.2) circle (.08cm);
        \fill[black] (0,0) circle (.08cm);
        \fill[black] (3.2,0) circle (.08cm);
        \node[left,black] at (.83, 1.24) {\tiny $0$};
        \node[right,black] at (2.4, 1.24) {\tiny $\infty$};
        \node[above] at (1.6, .73) {\tiny $E_1$};
        \node[above] at (1.6, .37) {\tiny $E_2$};
        \node[below] at (1.6, -.78) {\tiny $E_h$};
        \node[thick] at (1.6, .1) {\vdots};
      \end{tikzpicture}
    }
  \end{center}
  As the regular model $\cX$ is obtained by replacing the $A_2$- and
  $A_3$-singularities by chains of lines, the extended graphs of
  $\cX_s$ and $\cX'_s$ are homotopy equivalent, and so we may restrict
  to $\cX'_s$. On the special fibre
  $\cX_s'$, $W_p$ interchanges the two irreducible components. Recall
  also that the supersingular points $\SPS_1$ are $\mathbb{F}_{p^2}$-rational,
  and if $x\in\SPS_1$ is a supersingular point, corresponding to the
  class of a supersingular elliptic curve $E/\FFpbar$, then
  $W_p(x)=x^{(p)}$ is the point corresponding to
  $E^{(p)}=E/\ker(F)$. So the automorphism $W_p$ extends to to an
  automorphism of the graph, fixing $v_0$ and interchanging $Z_0$ and
  $Z_\infty$, and mapping the edges labelled $E_i$ to $E_i^{(p)}$. The
  homology $H_1(\Gamma_{\cX'_s,\Sigma},\ZZ)$ is freely generated
  by the cycles $\gamma_i=(0)+(E_i)-(\infty)$, and $W_p\colon \gamma_i
  \mapsto -\gamma_i^{(p)} = -(0)-(E_i^{(p)})+(\infty)$. Now the only
  element of $\overline{\Hom}(E_i,E_j)_p$ is the Frobenius $E_i\to
  E_i^{(p)}=E_j$, and therefore the matrix of $T_p=-W_p$ equals ${}^tB(p)$. 
\end{proof}
  
\subsection{Component groups}

Throughout this section, we assume that $p>3$. Let $N=p^rM$, with
$(p,M)=1$ and $r\ge1$. As in the previous section, work over $S$, the
strict henselisation of $\Spec \ZZ_{(p)}$. Let $\fm$ be a reduced
modulus on $X_0(N)$ supported at the cusps, and $\mathbb{T}$ a
subalgebra of the Hecke algebra $\ZZ[\{T_\ell\}]$ which preserves the
support of $\fm$. Let $J_\fm$ be the generalized Jacobian of $X_0(N)$
for the modulus $\fm$. Then $\mathbb{T}$ acts on $J_\fm$, stabilising
the torus $T_\fm$. It therefore acts on the extension of component
groups
\[
  0 \to \comp{T_\fm} \to \comp{J_\fm} \to \comp{J} \to 0
\]
and the action commutes with the action of $\Gal(\FFpbar/\FFp)$.
For the action of $\mathbb{T}$ on $\comp{J}$ we have the following
result, proved by Edixhoven \cite{Edi91}, generalizing Ribet
\cite{Rib88} who treated the case of $N$ squarefree.

\begin{thm}
  For every $\ell\notdivides N$, $T_\ell$ acts on $\comp{J}$ as
  multiplication by $\ell+1$.
\end{thm}

\begin{cor}
  Assume that $M$ is squarefree and $p>3$. Then for every
  $\ell\notdivides N$, $T_\ell$ acts on $\comp{J_\fm}$ as multiplication
  by $\ell+1$.
\end{cor}

\begin{proof}
  Let $x\simeq\Spec\QQ(\mmu_{(d,N/d)})\subset X_0(N)^\infty$ be a cusp,
  where $d\divides N$. As $M$ is squarefree, $(d,N/d)$ is a power of
  $p$. Therefore $\Gal(\overline{\QQ}/\QQ)$ acts trivially on
  $\comp{T_\fm}$ by \eqref{eq:pi0-of-tori}. So $T_\ell$ acts on
  $\comp{T_\fm}$ as multiplication by $\ell+1$. So the endomorphism
  $T_\ell-\ell-1$ of $\comp{J_\fm}$ factors through a map
  $\comp{J}\to\comp{T_\fm}$, which is zero as $\comp{J}$ is finite and
  $\comp{T_\fm}$ is free. 
\end{proof}

\begin{rem}
  Similarly, let $N$ be arbitrary, and $\fm$ the reduced modulus on
  $X_0(N)$ which is the sum of all the cusps. Write
  $T_\fm \to T_\pspl$ for the maximal quotient which is split over
  $\QQ(\mmu_{p^r})$. Let $J_\pspl$ be the corresponding quotient of
  $J_\fm$. Then by Corollary \ref{cor:tori-exact} the sequence of
  N\'eron models
  \[
    0\to \cT_\pspl \to \cJ_\pspl  \to \cJ \to 0
  \]
  is exact, and the same argument shows that $T_\ell=\ell+1$ on
  $\comp{\cJ_\pspl}$.
\end{rem}

Now we turn to the abelian group structure of $\comp{J_\fm}$.

For $N=pM$, $(p,M)=1$, the structure of $\comp{J}$ was determined
completely by Deligne, and described by Mazur and Rapoport in
\cite{MR77}, using the description of the regular model of $X_0(N)$
given in \cite{DeRa72} --- see Table 2 on p.~174 and the calculations of \S2
in \emph{loc.~cit.}, and the corrections to their calculations made by
Edixhoven \cite[4.4.1]{Edi91}. We recall these formulae in
\ref{cor:mazur-rapoport-formula} below.

For general $N$, the minimal desingularisation  $\cX \to \cX'$ was computed
by Edixhoven \cite{Edi90} using the description of $\cX'$ in
\cite{KM85}. Since the component of $\cX_s'$ meeting the cusp $\infty$
has multiplicity one, $\cX$ satisfies hypotheses (H1--3).
From this it is in principle an exercise to compute
$\comp{J}$ in any given case, and in \cite[4.4.2]{Edi91} this is done for
$N=p^2$.

We will compute $\comp{J_\fm}$ in various cases. First some notation:
as in the previous section, let $\SPSM\subset X_0(M)(\FFpbar)$ be the set
of supersingular points, and $n=\#\SPSM$.  For $j\in\{2,3\}$ let $\e j$
be the number of elements $(E,C)\in \SPSM$ for which
$\#\mathrm{Aut}(E,C)=2j$. 

\subsubsection{$X_0(pM)$ with $(p, M)=1$ and $\fm=(\infty)+(0)$}
\label{sec:X0pM}

\begin{thm}
\label{thm:X0pM}
  Let $N=pM$ with $(p,M)=1$. Let $J_\fm$ be the
  generalized Jacobian of $X_0(N)$ with respect to the modulus
  $\fm=(\infty)+(0)$. Then:
  \begin{enumerate}[\upshape (a)]
\item
\label{thm:X0pM1}
   $\comp{J_\fm} \simeq \ZZ \oplus (\ZZ/2\ZZ)^{\max(\e2-1,0)}
    \oplus (\ZZ/3\ZZ)^{\max(\e3-1,0)}$
\item
\label{thm:X0pM2}
  The homomorphism $\comp{T_\fm}=\ZZ \to \comp{J_\fm}$ is
    given in terms of the isomorphism \textup{(a)} by
    \[
      1\mapsto \begin{cases}
        n               &\text{if }\e2=\e3=0 \\    
        (2n-\e2 ; 1,\dots,1) &\text{if }\e2>0,\ \e3=0 \\
        (3n-2\e3; 1,\dots,1) &\text{if  }\e2=0,\ \e3>0\\
        (6n-3\e2-4\e3;1,\dots,1;1,\dots,1)
                         &\text{otherwise.}          
      \end{cases}
    \]
  \end{enumerate}
\end{thm}

\begin{proof}
  Recall that the special fibre $\cX_s'$ of the Deligne--Rapoport
  model of $X_0(N)$ is the union of two copies of $X_0(M)_{\FFpbar}$,
  meeting tranversally at the supersingular points. The cusps $\infty$
  and $0$ belong to different components. The total space $\cX'$ has a
  type $A_j$ quotient singularity at each point where
  $\#\Aut(E,C)=2j\in\{4,6\}$. Taking their minimal resolution gives the
  model $\cX$. Its special fibre is obtained by replacing each
  crossing point which is an $A_j$-singularity with a chain of $(j-1)$
  copies of $\mathbb{P}^1$. In other words, $\cX_s$ has $2+\e2+2\e3$
  irreducible components:
  \begin{itemize}
  \item $Z_\infty$ and $Z_0$, the strict transforms of the irreducible
    components of $\cX_s'$, isomorphic to
    $X_0(M)_{\FFpbar}$, and labelled in such a way that the cusp
    $\alpha\in\{\infty,0\}$ belongs to $Z_\alpha$.
  \item Components in the fibres of $\cX_s\to \cX_s'$: denote these as
    $E_i$ (for $1\le i\le \e2$), and $F_{\infty,i}$, $F_{0,i}$ (for
    $1\le i\le \e3$), where $F_{\alpha,i}$ intersects $Z_\alpha$.
  \end{itemize}
  Their intersection numbers are
  \begin{itemize}
  \item $(Z_\alpha.Z_\alpha)=-n$, $(Z_\infty.Z_0)=n-\e2-\e3$
  \item $(Z_\alpha.E_i)=1=(Z_\alpha.F_{\alpha,i})$
  \item All other intersection numbers are zero.
  \end{itemize}
  We now use the formula for $\comp{J_\fm}$ from Theorem \ref{thm:phi_m}. We
  have $C=\pi_0(\widetilde\cX_s)$, and write $\dual Y\in \ZZ^C$ for the
  basis element dual to $Y\in C$. 
          We also have $I=\{\infty,0\}$, and $e$ (as in Theorem
  \ref{thm:phi_m}) equals $(1,1)$. Therefore
  $\ZZ^I/e\ZZ =\ZZ.V$ where $V$ is the image of the dual of
  $\infty$ (so that $-V$ is the image of the dual of $0$). The map
  $h\colon \ZZ[C] \to \ZZ^I/e\ZZ$ takes $Z_\infty$ to $V$ and $Z_0$ to
  $-V$, and all other elements of $C$ to $0$. The map $(a,h)\colon \ZZ[C]
  \to \ZZ^C\oplus \ZZ^I/e\ZZ$ is then given by the matrix:
  \[
  \begin{blockarray}{cccccccccccc}
    & Z_\infty & Z_0 & E_1 & \cdots & E_{\e2} & F_{\infty,1} & F_{0,1}
    &\cdots &F_{\infty,\e3}&F_{0,\e3}\\ 
    \begin{block}{c[ccccccccccc]}
      \dual Z_\infty & -n & n-\e2-\e3& 1 & \cdots & 1 & 1 & 0 & \cdots &1 &0
      \\
      \dual Z_0 & n-\e2-\e3& -n & 1 & \cdots & 1 & 0 & 1 & \cdots &0 &1
      \\
      \dual E_1 & 1 & 1 & -2 &
      \\
      \vdots & \vdots &\vdots&& \ddots
      \\
      \dual E_{\e2}&1&1&&&-2
      \\
      \dual F_{\infty,1}&1&0&&&&-2&1
      \\
      \dual F_{0,c_1}&0&1&&&&1&-2
      \\
      \vdots&\vdots&\vdots& &&&&&\ddots
      \\
      \dual F_{\infty,\e3}&1&0&&&&&&&-2&1
      \\
      \dual F_{0,\e3}&0&1&&&&&&&1&-2
      \\
      V&1&-1&0&\cdots&&&&&&0
      \\
    \end{block}
  \end{blockarray}
  \]
  As a basis for $\ZZ^{C,0}$ we take $\overline Z=\dual Z_\infty -
  \dual Z_0$,   $\overline E_i=\dual E_i - \dual Z_0$,
  $\overline F_{\alpha,i}= \dual F_{\alpha,i}- \dual Z_0$. So $\comp{J_\fm}$ is
  isomorphic to the quotient of the free module generated by
  $\overline Z$, $\{\overline E_i\}$, $\{\overline F_{\alpha,i}\}$ and
  $V$ by the submodule of relations
  \begin{gather*}
    \overline Z = 2 \overline E_i = 3\overline F_{0,i},\qquad
    \overline F_{\infty,i} = 2\overline F_{0,i} \\
    V = n\overline Z - \sum_{i=1}^{\e2} \overline E_i -
    2\sum_{i=1}^{\e3} \overline F_{0,i}\,.
  \end{gather*}
  If $\e2>1$ then for every $i>1$, $U_i=\overline E_i-\overline E_0$
  has order $2$, and if $\e3>1$ then $V_i=\overline F_{0,i}-\overline
  F_{0,0}$ has order $3$. The subgroup generated by $\overline Z$,
  $\overline E_1$ (if $\e1\ge1$) and $\overline F_{1,0}$ (if
  $\e3\ge1$) is infinite cyclic, with generator
  \begin{align*}
    \overline Z\ \ \qquad                &\text{if }\e2=\e3=0 \\
    \overline E_1\ \qquad                &\text{if }\e2>0,\ \e3=0 \\
    \overline F_{0,1}\qquad              &\text{if  }\e2=0,\ \e3>0\\
    \overline E_1- \overline F_{0,1}\quad&\text{otherwise}          
  \end{align*}
  This gives (a), and (b) follows since the inclusion
  $\comp{T_\fm}=\ZZ \to \comp{J_\fm}$ maps $1$ to $V$.
\end{proof}

Since $\comp{J}=\comp{J_\fm}/\comp{T_\fm}$, an easy computation gives:
\begin{cor}\textup{(\cite[Table 2]{MR77}; \cite[4.4.1]{Edi91})}
  \label{cor:mazur-rapoport-formula}
\[
  \comp{J}\simeq \ZZ/P\ZZ \oplus (\ZZ/2\ZZ)^{\max(\e2-2,0)}\oplus
  (\ZZ/3\ZZ)^{\max(\e3-2,0)}
\]
where
\[
  P = 2^{\min(\e2,2)}3^{\min(\e3,2)}\left(n - \frac{\e2}2 -
  \frac{2\e3}3 \right).
\]
\end{cor}

\subsubsection{$X_0(p)$ with $\fm=(\infty)+(0)$}
\label{sec:X0p}

In the setting of Theorem \ref{thm:X0pM}, if $\e2$ and $\e3$ are at
most $1$, then $\comp{J_\fm}$ is infinite cyclic and the map
$\comp{T_\fm}\simeq\ZZ \to \comp{J_\fm}$ is, up to sign,
multiplication by the order of the cyclic group $\comp{J}$. Therefore
the actions of all the Hecke operators $T_\ell$ (including for
$p\divides N$) can be computed from the actions on $\comp{T_\fm}$. For
example, suppose $N=p$. Then by Example
\ref{ex:hecke-calcs}(\ref{ex:hecke-calcs-p}), without appealing to the
results of Ribet and Edixhoven we obtain:
\begin{cor}
  Suppose that $N=p$ and $\fm=(\infty)+(0)$. Then
  \begin{itemize}
  \item $\comp{J_\fm}$ is infinite cyclic.
  \item $\comp{J}=\coker(\comp{T_\fm} \to \comp{J_\fm})$ is cyclic of
    order $n$, the numerator of $(p-1)/12$.
  \item For $\ell\ne p$, $T_\ell=\ell+1$ on $\comp{J_\fm}$, and
    $T_p=1$ on $\comp{J_\fm}$.
  \end{itemize}
\end{cor}

\subsubsection{ $X_0(pM)$ with $(p,M)=1$ and $\fm$ a general cuspidal
modulus}
\label{sec:X0pM-bis}
Now let $\fm$ be any nonzero reduced modulus supported on the
cusps of $X_0(N)$, with $p$ exactly dividing $N$. Recall that we are
working over the strict henselisation $R$ of $\ZZ_{(p)}$. Then since $p^2\notdivides
N$, all the cusps are rational over $\F$ so $e=(1,\dots,1)$ and
\[
  \comp{J_\fm} = \coker( \ZZ[C] \xrightarrow{(a,h)} \ZZ^{C,0}\oplus
  \ZZ^I/\diag(\ZZ) )
\]
with $I=\supp(\fm)\subset X_0(N)^\infty(\QQbar)$.

\begin{prop}
    If the closure of the support of $\fm$ meets just one component of the
    special fibre $\cX'_s$, then there is a canonical splitting
    \[
      \comp{J_\fm}=\comp{J}\oplus \comp{T_\fm}.
    \]
    Otherwise, if $x_0=\infty$, $x_0\in\supp(\fm)$ meet $Z_\infty$,
    $Z_0$ respectively, and $\fm'=(x_\infty)+(x_0)$, then there
    is a canonical splitting 
    \[
      \comp{J_\fm}=\comp{J_{\fm'}}\oplus
      \ZZ^{I\setminus\{x_\infty,x_0\}}.
    \]
\end{prop}
\begin{proof}
  In the first case, we may assume that the closure of the support of $\fm$ meets only
  the component $Z_\infty$.   Then $h(Y)=0$ if $Y\in C$, $Y\ne
  Z_\infty$ but $h(Z_\infty)=(1,\dots,1)$, so the composite $h\colon
  \ZZ[C]\to\ZZ^I/\diag(\ZZ)$ is zero.

  In the second case, we have $h(Y)=0$ if $Y\notin\{Z_\infty,Z_0\}$,
  and
  \begin{align*}
    h(Z_\infty)&= (1,\dots,1,0,\dots,0) \\
    h(Z_0) &= (0,\dots,0,1,\dots,1)
  \end{align*}
  for a suitable ordering of $I$. Therefore
  \[
    \ZZ^I/\diag(\ZZ) = \im(h) \oplus \{ b \in \ZZ^I \mid
    b(Z_\infty)=b(Z_0) \}/\diag(\ZZ)
  \]
  giving the splitting.
\end{proof}

\subsubsection{$X_0(p^2)$}
\label{sec:X0p2}

Finally, let us consider the curve $X_0(p^2)$, $p>3$. The Katz-Mazur model
$\cX'$ over $S$ has three irreducible components in its special fibre, which
we denote $Z_i'$ ($0\le i\le2$). The non-supersingular non-cuspidal points of $Z_i'$
parametrize pairs $(E,C)$, where $E$ is an elliptic curve and $C$ is
a cyclic (in the sense of Drinfeld) subgroup scheme of rank $p^2$,
whose \'etale quotient has rank $p^i$. The components $Z_0'$, $Z_2'$
have multiplicity $1$, and $Z_1'$ has multiplicity $p-1$. They meet at
the supersingular points.

The cuspidal divisor $X_0(p^2)_\QQ^\infty$ consists of three closed
points $\infty=z_1 = \Spec \QQ$, $z_p=  \Spec\QQ(\mmu_p)$ and
$0=z_{p^2}= \Spec\QQ$, in the notation of Section \ref{sec:mod-gen}.
For each $i$, the closure in $\cX'$ of the point  $z_{p^i}$ meets the
component $Z_i'$ in a single point, and the
completed local ring at the intersection is computed in
\cite[Proposition 1.2.2.1]{Edi90} as
\begin{align}
  \ZZ_p[[q]]\qquad & \quad\text{if $i=0$} \notag \\
  \ZZ_p[\mmu_p][[q]]\quad &\quad\phantom{\text{if $i=$ }}1
                            \label{eq:completed-local-p2}\\
  \ZZ_p[[q^{1/p^2}]]\quad& \quad\phantom{\text{if $i=$ }} 2 \notag
\end{align}
where $q$ is the usual parameter at infinity on the modular curve of
level $1$.

The minimal resolution $\pi\colon \cX \to \cX'$ is described in detail
in \cite[\S1.5]{Edi90}.  We summarise the final result. Write
$p=12k+1+4a+6b$, with $a$, $b\in\{0,1\}$. We again work over the strict henselisation $R$ of $\ZZ_{(p)}$.

The Katz--Mazur model $\cX'$ has
exactly two singular points, which are the points $x_0$,
$x_{1728} \in Z_1'$ lying over the points $j=0$, $1728$ in the curve
$X(1)_{\FFpbar}$. Let $E=\pi^{-1}(x_{1728})^{\red}$,
$F=\pi^{-1}(x_0)$. Then $E\simeq F \simeq \mathbb{P}^1$, $E$ has multiplicity
$(p-1+2b)/2$ and $F$ has multiplicity $(p-1+2a)/3$.

Let $Z_i$ be the reduced strict transform of $Z_i'$. The intersection
matrix of $\cX_s$ is:
\begin{equation}
  \label{eq:X0p2-int-mat}
  \begin{blockarray}{cccccc}
     & Z_0 & Z_1 & Z_2 & E & F\\ 
    \begin{block}{c[ccccc]}
      Z_0 & -\M & k & k & b & a \\
      Z_1 & k & -1 & k & 1 & 1 \\
      Z_2 & k & k & -\M & b & a \\
      E & b & 1 & b & -2 & 0 \\
      F  & a & 1 & a & 0 & -3 \\
    \end{block}
  \end{blockarray}
\end{equation}
where $\M=(p^2-1)/12-k$. As a basis for $\ker(\ZZ^C \xrightarrow{b}
\ZZ)$ we take $\overline Y = \dual Y - d_Y \dual Z_2$, for $Y\in \{Z_0, Z_1,
E, F\}$ and where $d_Y$ is the multiplicity of $Y$. (Since the residue
field is perfect, $d_Y=\delta_Y$.)

We first consider the modulus $\fm = X_0(p^2)_\QQ^\infty=\sum_{0\le i\le 2}(z_{p^i})$ of all cusps.  Since the cusp $z_p$ is isomorphic
to $\Spec \QQ(\mmu_p)$, and the other cusps are rational,
$e=(1,p-1,1)$. From the description
\eqref{eq:completed-local-p2} of the completed local rings at the
cusps, we see that $\Sigma\simeq \Spec R \sqcup \Spec R[\mmu_p]\sqcup
\Spec R$, and the pullback of the divisor $Z_i$ to the component of
$\Sigma$ which it meets has degree $1$. Therefore the matrix
$(h_{ij})$ giving
the pairing $C\times I\to \ZZ$ in Theorem
\ref{thm:phi_m} is
\[
  \begin{blockarray}{cccccc}
     & Z_0 & Z_1 & Z_2 & E & F\\ 
    \begin{block}{c[ccccc]}
      z_1 & 1 & 0 & 0 & 0 & 0 \\
      z_p & 0 & 1 & 0 & 0 & 0 \\
      z_{p^2} & 0 & 0 & 1 & 0 & 0 \\
    \end{block}
  \end{blockarray}\,.
\]
 Let $V_i\in \ZZ^I/e\ZZ$ be the
image of the $i$-th basis vector (dual to $z_{p^i}$) of
$\ZZ^I$. We will take $\{V_0,V_1\}$ as basis for $\ZZ^I/e\ZZ$.

Next consider the modulus $\fm'=(\infty)+(0)=(z_1)+(z_{p^2})$. Then
$e=(1,1)$, and the pairing $C\times I \to \ZZ$ is given by the same
matrix with the $z_p$-row deleted, and $\ZZ^I/e\ZZ$ is generated by
$V_0=-V_2$.

Under the isomorphism of \ref{thm:phi_m}, the image of $\ZZ^I/e\ZZ$
in the homology of the complex \eqref{eq:raynaud} is the subgroup
$\comp{T_\fm}$ of $\comp{J_\fm}$. The analogous statement holds for $\fm'$.

\begin{thm}
  \textup{(a)} The component group $\comp{J_\fm}$ is isomorphic to $\ZZ^2$, and
  \textup(for a suitable choice of isomorphism\textup), the image of the generators
  $V_0$, $V_1$ of $\comp{T_\fm}$ are
  \[
    (\M+(3b-2a)k-a+b, -6k-2a-3b)\quad\text{and}\quad (-k-b,1).
  \]
  \textup{(b)} The component group $\comp{J_{\fm'}}$ is isomorphic to $\ZZ$, and
  \textup(up to sign\textup), the image of the generator
  $V_0$ of $\comp{T_\fm}$ is $(p^2-1)/24$.
\end{thm}

\begin{rem}
  (i) From the computation in (b) we recover the result
  \cite[Sect. 4.1, Prop. 2]{Edi91} that $\comp{J}$ is cyclic of order
  $(p^2-1)/24$.
  
  (ii) In both cases the map $\comp{T_\fm} \to \comp{J_\fm}$
  is an injection of free abelian groups of the same rank, so the
  action of Hecke operators on $\comp{T_\fm}$ determines that on
  $\comp{J_\fm}$ and therefore on the quotient $\comp{J}$, ``by pure
  thought''.
      \end{rem}  

\begin{proof}
  From \eqref{eq:X0p2-int-mat} we see that $\comp{J_\fm}$ is generated
  by $\{V_0,V_1,\overline Z_0,\overline Z_1,\overline E,\overline F\}$
  with relations
  \begin{gather*}
    V_0=\M\overline Z_0 -k\overline Z_1 -b\overline E-a\overline F\\
    V_1=-k\overline Z_0 +\overline Z_1 -\overline E -\overline F\\
    b\overline Z_0 +\overline Z_1-2\overline E=0=a\overline Z_0
    +\overline Z_1 -3 \overline F
  \end{gather*}
  and linear algebra then gives an isomorphism $\comp{J_\fm} \isom
  \ZZ^2$ by
  \begin{align*}
    \overline Z_0 &\mapsto (1,0) \\
    \overline Z_1 &\mapsto (2a-3b,6) \\
    \overline E &\mapsto (a-b,3) \\
    \overline F &\mapsto (a-b,2) \\
    V_0 &\mapsto (\M-(2k+1)a+(3k+1)b, -6k-2a-3b) \\
    V_1 &\mapsto (-k-b,1)
  \end{align*}
  This proves (a). For (b), we compose with the map $\ZZ^2
  \xrightarrow{(1,k+b)} \ZZ$, whose kernel is the subgroup generated
  by $V_1$, and which takes $V_0$ to
  \[
    \M-(2k+1)a+(3k+1)b +(k+b)(-6k-2a-3b)=\frac{p^2-1}{24}.
  \]
\end{proof}

\bibliography{neron}

\end{document}